\documentclass[]{article}
\usepackage[utf8]{inputenc}
\usepackage[english]{babel}
\usepackage{csquotes}		
\usepackage{geometry}
\title{Epsilon-regularity for the Brakke flow with boundary}
\author{Carlo Gasparetto\thanks{Dipartimento di Matematica, Università di Pisa, Pisa (Italy). \href{mailto:carlo.gasparetto@dm.unipi.it}{carlo.gasparetto@dm.unipi.it} }}
\date{}

\usepackage[style = numeric, backend=biber]{biblatex}
\usepackage{mathtools,
	amsthm,
	amssymb		
}

\numberwithin{equation}{section}						
\renewenvironment{proof}[1][\proofname]{{\par\medskip\noindent\bfseries #1. }}{\qed\par}		

\newcommand{\RR}{{\mathbb R}}
\newcommand{\NN}{{\mathbb N}}
\newcommand{\Lc}{\mathcal{L}}
\newcommand{\Mc}{\mathcal{M}}
\newcommand{\Vc}{\mathcal{V}}
\newcommand{\Dc}{\mathcal{D}}
\newcommand{\Hc}{\mathcal{H}}
\newcommand{\Mb}{\mathbf{M}}
\newcommand{\Nb}{\mathbb{N}}
\newcommand{\weakly}{\rightharpoonup}
\newcommand{\graph}{\mathop{\rm graph}}
\newcommand{\compact}{\subset\subset}
\newcommand{\rest}{\llcorner}
\DeclareMathOperator{\supp}{supp}
\DeclareMathOperator{\dist}{dist}
\DeclareMathOperator{\Int}{Int}
\DeclareMathOperator{\trace}{trace}
\DeclareMathOperator{\dive}{div}
\DeclareMathOperator{\Span}{span}
\newcommand{\push}{_\sharp}
\newcommand{\e}{\mathbf{e}}
\newcommand{\eps}{\varepsilon}		

\let\oldphi\phi						
\newcommand{\phin}{\oldphi}
\renewcommand{\phi}{\varphi}

\let\emptyset\varnothing

\newcommand{\de}{\partial}
\newcommand{\compl}{c}			
\newcommand{\pt}[1]{\left( #1 \right) }
\newcommand{\norm}[1]{\left\| #1 \right\|}
\newcommand{\iprod}[2]{\left\langle #1 ; #2 \right\rangle}			
\renewcommand{\hat }{\widehat}
\renewcommand{\tilde }{\widetilde}

\newcommand{\Tint}[1]{(#1]}			
\newcommand{\BF}{\mathcal{BF}_m}			
\newcommand{\IM}{\mathcal{IM}}		
\newcommand{\Var}{\mathrm{Var}}		
\newcommand{\hk}{\Psi}			
\DeclareMathOperator{\Clos}{Clos}		

\newcommand{\Flat}{\mathcal{F}}			

\newcommand{\sprod}{\colon}
\newcommand{\trasp}{T}

\newcommand{\Gr}{\mathrm{Gr}}
\DeclareMathOperator{\mdr}{mdr}

\usepackage{enumerate}

\newcommand{\haus}{\Hc}				
\DeclareMathOperator{\osc}{osc}
\newcommand{\bound}{K}

\newcommand{\Id}{I}
\newcommand{\pball}{\mathcal{B}}

\usepackage[hypertexnames=false]{hyperref}		
\usepackage{cleveref}		
\crefname{subsection}{Subsection}{Subsections}		

\theoremstyle{plain}

\newtheorem{theorem}{Theorem}[section]
\newtheorem{proposition}[theorem]{Proposition}
\newtheorem{lemma}[theorem]{Lemma}
\newtheorem{corollary}[theorem]{Corollary}

\newtheorem{definition}[theorem]{Definition}
\theoremstyle{definition}
\newtheorem{remark}[theorem]{Remark}

\begin{document}
\maketitle
\begin{abstract}
	We prove that, if a Brakke flow with boundary is close enough to a stationary half-plane with density one, then it is $C^{1,\alpha}$. Our approach is based on viscosity techniques introduced by Savin in the context of elliptic equations. The same techniques can be used to give a proof of Brakke's (interior) regularity theorem which is alternative to the original one.
\end{abstract}
\section{Introduction}
In this paper, we state and prove a Brakke-type theorem for the mean curvature flow with boundary, that is a flow of $m$-dimensional surfaces in $\RR^d$ so that at every point the normal component of the velocity is equal to the mean curvature and the boundary is fixed. A weak notion of such a flow has been recently introduced in \cite{whitebdry} by using integral varifolds, as devised by Brakke \cite{brakke}. The objects in question are called \textit{integral Brakke flows with boundary}.

In short, given a $(m-1)$-dimensional submanifold $\Gamma$, an integral Brakke flow with boundary $\Gamma$ is a collection $\{V_t\}_{t\in I}$ of $m$-dimensional integral varifolds with the constraint that the first variation of $V_t$ is a measure whose singular part with respect to $||V_t||$ behaves like $\haus^{m-1}\rest\Gamma$ and the varifolds satisfy an evolution equation that encodes the information on the velocity. A precise definition will be given in \Cref{sec:notation}.

The main result of this paper (\Cref{thm:final}) is that, if a Brakke flow in a ball of radius $1$ is close enough (in some appropriate topology) to a unit-density half plane (which is a stationary solution to the mean curvature flow with a prescribed straight boundary), then the Brakke flow becomes smooth up to the boundary in a smaller ball and after some fixed waiting time. Roughly stated, the main result is the following
\begin{theorem}[$\eps$-regularity]\label{thm:main_intro}
	Let $\Gamma$ be a $C^{1,\alpha}$, $(m-1)$-dimensional submanifold of $B_1$ and let $\{V_t\}_{t\in[-\Lambda,0]}$ be an integral Brakke flow with boundary $\Gamma$ in $B_1\times[-\Lambda,0]$. Assume the following:
	\begin{enumerate}
		\item \label{ass:disap} $0\in\supp||V_0||$;
		\item \label{ass:mass} at time $t=-\Lambda$, the mass measure $||V_t||$ is close to that of a $m$-dimensional half disk;
		\item \label{ass:flat}there exists a half-plane $S^+$ such that, for every $t\in[-\Lambda,0]$:
		\begin{equation*}
		\supp \norm{V_t}\subset\{x\in \RR^d\colon \dist(x,S^+)\le\eps\}.
		\end{equation*}
	\end{enumerate}
	If $\eps$ and $\Lambda$ are small enough, then there exist small constants $\eta,\beta$ and a family $\{N_t\}_{t\in\Tint{-\eta^2,0}}$ of $C^{1,\beta}$ surfaces with boundary $\Gamma$ such that
	\begin{equation*}
	\supp \norm{V_t} \cap B_\eta = N_t
	\end{equation*}
	for every $t\in\Tint{-\eta^2,0}$.
\end{theorem}

We briefly comment on the assumptions. The key assumptions are \Cref{ass:mass} and \Cref{ass:flat}, which describe how the Brakke flow is close to being a half-plane (with a straight boundary). \Cref{ass:disap}, on the other hand, prevents a \enquote{pathological} behavior of Brakke flows, which is the possibility of a sudden loss of mass (see, for example, \cite[Section 2.3]{Tonegawabook}). The statement and the assumptions will be made more rigorous in Sections \ref{sec:IOF} and \ref{sec:final}


A central point in our work is that, under appropriate assumptions, the support of an integral Brakke flow with boundary satisfies a maximum principle. In order to fix ideas, assume that the support of the flow is the graph of some function $u:\RR^m\to\RR^{d-m}$. Then it can be proved that $|u|$ is a viscosity subsolution (in a suitable sense which we will describe at a later stage) to
\begin{equation*}
\de_t\phi-\Mc^+(D^2\phi)\le0,
\end{equation*} 
where $\Mc^+$ is a Pucci maximal operator.
We may therefore exploit this property to adopt a technique developed by Savin in \cite{savin07} in the framework of elliptic equations and later adapted by Wang in \cite{wangsmall} to parabolic equations, which we now summarize in our case.
The key step in proving \Cref{thm:main_intro} is proving the following \textit{improvement of flatness}:
\begin{proposition}[Improvement of flatness]\label{prop:IOF_intro}
	Under the assumption of \Cref{thm:main_intro}, there exist $\eta>0$ and a half plane $T^+$ close to $S^+$ such that, for every $t\in(-\eta^2,0]$,
	\begin{equation*}
	\supp \norm{V_t}\cap B_\eta\subset\bigg\{x\in \RR^d\colon \dist(x,T^+)\le\frac{\eps}{2}\eta\bigg\}.
	\end{equation*}
\end{proposition}
In summary, if the Brakke flow is \enquote{$\eps$-flat} at scale $1$, then it becomes \enquote{$\eta\eps/2$-flat} at scale $\eta$, for some $\eta$ small and universal; from this, proving $C^{1,\alpha}$-regularity is classical.

The proof of \Cref{prop:IOF_intro} is based on a contradiction and compactness argument. Assume one can find a sequence of flatter and flatter Brakke flows for which the conclusion of \Cref{prop:IOF_intro} does not hold. Then appropriate rescalings of the supports of such flows converge in a suitable sense to the graph of a solution to the heat equation. The desired improvement of flatness is a straightforward consequence of classical Schauder estimates. 
The above convergence is achieved via a Harnack-type inequality, in the spirit of \cite{wangsmall}, and a barrier argument that describes the behavior of the Brakke flow near the boundary.

\Cref{thm:main_intro} answers a question left open in \cite[Remark 11.2]{whitebdry}, that is whether an integral Brakke flow with boundary that has a tangent flow which is a unit-density half-plane is smooth in a backward neighborhood. The reader should also compare our results with the regularity theorems proved in \cite{whitebdry}. The latter are proved under the additional assumption that the flow is \textit{standard}: namely the flow has to be smooth at every point where a tangent flow is a unit-density half-plane (see \cite[Definition 11.1]{whitebdry}). Since we only prove backward regularity, our result does not guarantee (as it should not be expected) that an integral Brakke flow with boundary satisfying the assumptions of \Cref{thm:main_intro} is actually standard.

The first $\eps$-regularity theorem for the mean curvature flow (without boundary) was proved in \cite{brakke} and then refined in \cite{kasai_tonegawa}, where the authors extended the result to mean curvature flow in general ambient manifolds. Another relevant reference is the recent work \cite{stuvard_Endtime}. The above mentioned proofs are variational and rely on $L^2$ energy estimates, somehow in the spirit of \cite{allard1972}. We think that a variational proof of \Cref{thm:main_intro} may be performed, by adapting the arguments in \cite{allard1975} and in \cite{bourni2016allard} to account for the presence of the boundary. As mentioned, our proof is based on an argument first developed in \cite{savin07} for elliptic equations and then adapted to parabolic equations in \cite{wangsmall}. This method was used in \cite{savin2017viscosity} to prove an Allard-type theorem for minimal surfaces. Although an adaptation of the same techniques to the mean curvature flow seems quite natural, to the best of the author's knowledge this paper is the first instance in which these techniques are used for the mean curvature flow.

The regularity of mean curvature flow with boundary has been briefly investigated also in \cite{Whi95, Whi05}. One should also see \cite{Edelen}, where the author defines a Brakke flow with a free boundary condition. Another definition of Brakke flow with fixed boundary has been investigated in \cite{stuvard2021existence}.

\subsection{Structure of the paper}
In \Cref{sec:notation}, we collect some notations that will be used throughout the paper and some well known facts about rectifiable measures. We then recall the definition of Integral Brakke flow with boundary, as stated in \cite{whitebdry}.

\Cref{sec:properties} is dedicated to collecting some known results about Integral Brakke flows and to adapting them to the case of an integral Brakke flow with fixed boundary. 

The core of the paper is \Cref{sec:IOF}, where we state and prove the improvement of flatness described in \Cref{prop:IOF_intro}, which will later yield the desired $C^{1,\beta}$ regularity. The proof of this result is described in \Cref{subsec:iof_proof}. The aforementioned barrier argument and Harnack-type inequality, which are crucial for obtaining the desired compactness, are discussed in \Cref{subsec:bdry} and \Cref{subsec:osc_decay}, respectively. 

Finally, the proof of \Cref{thm:main_intro} is given in \Cref{sec:final}, where we iterate the improvement of flatness to obtain the desired regularity.

\subsection*{Acknowledgements}
The author is extremely grateful to Guido De Philippis for his crucial help in writing this paper and to Luigi De Masi for countless discussions and comments.
This material is based upon work partially supported by the European Research Council (ERC), under the
European Union’s Horizon 2020 research and innovation program, through the project ERC VAREG
- Variational approach to the regularity of the free boundaries (grant agreement No. 853404) and by INDAM-GNAMPA.
This work was completed when the author was a PhD student at SISSA, Trieste.

\section{Preliminaries, notation and definitions}\label{sec:notation}
Throughout the paper, we consider fixed two positive integers $m$ and $d$ such that $m\le d$.
All the constants taken in consideration in the present work depend, in general, on these two parameters, although we will mostly avoid stating such dependency.

For the present section, we introduce two generic positive integers $k\le n$ to define some objects in full generality.

\subsection{Space-time}
By $\RR^{n,1}$ we denote the space $\{(x,t)\colon x\in\RR^n\mbox{ and }t\in\RR\}$. We use upper-case letters to denote points in $\RR^{n,1}$, for example $X = (x,t)$.

For any couple $X=(x,t)$ and $Y=(y,s)$ of points in $\RR^{n,1}$, we let
\begin{equation*}
	\rho(X,Y)=|x-y|+|t-s|^{1/2};
\end{equation*}
$\rho$ is a metric on $\RR^{n,1}$ (see, for example \cite[Exercise 8.5.1]{krylov_holder}) and the topology that $\rho$ induces on $\RR^{n,1}$ coincides with the euclidean topology of $\RR^{n+1}$.
In particular, if $d_H(E,F)$ is the Hausdorff distance between $E$ and $F$ with respect to $\rho$ and $K$ is a compact subset of $\RR^{n,1}$, then the space of non-empty closed subsets of $K$ is a compact metric space, when endowed with the metric $d_H$.

If $x\in\RR^n$ and $r>0$, we set $B^n_r(x) = \{y\in\RR^n\colon|y-x|<r\}$. When the dimension of the space is clear, we omit its indication and simply write $B_r(x)$. We also omit the indication of the center of the ball, whenever it coincides with $0$, so that $B_r=B_r(0)$.
If $(x,t)\in\RR^{n,1}$, we define the parabolic cylinder
\begin{equation}\notag
Q^n_r(x,t) = B^n_r(x)\times(t-r^2,t],
\end{equation}
where the apex $n$ indicates the dimension of the space component; as above, its indication will be omitted when no confusion shall arise. Lastly, we let $Q_r = Q_r(0,0)$.

$\de_p(U\times(a,b))$ denotes the \textit{parabolic boundary} of the cylinder $U\times (a,b)$, where $U\subset\RR^n$:
\begin{equation}\notag
\de_p(U\times (a,b)):=\big(\overline U\times\{a\}\big)\cup\big(\de U\times(a,b)\big).
\end{equation}

We define the measures $\Lc^{n,1}$ and $\haus^{s,1}$ (for any $0\le s \le n$) on $\RR^{n,1}$ by
\begin{equation}\notag
\Lc^{n,1}(E\times F) =\Lc^n(E)\times\Lc^1(F),\qquad \haus^{s,1}(E\times F) = \haus^s(E)\times\Lc^1(F)
\end{equation}
for $E\subset\RR^n$ and $F\subset\RR$, where $\Lc^n$ is the Lebesgue measure in $\RR^n$ and $\haus^s$ is the $s$-dimensional Hausdorff measure in $\RR^n$. 

For any function $f:\RR^{n,1}\to\RR^k$, we denote by $\nabla f(x,t)$ the gradient of the function $f(\cdot,t)$ computed at $x$ and by $\de_tf(x,t)$ the derivative of $f(x,\cdot)$ computed at $t$, whenever they are defined.

Lastly, for a set $E\subset\RR^{n}$ and $x\in\RR^n$, we let $\chi_E(x)=0$ if $x\notin E$ and $\chi_E(x)=1$ if $x\in E$.

\subsection{Linear functions and subspaces of the euclidean space}\label{subsec:subspaces}
We let $\{\e_1,\dots,\e_n\}$ be the canonical orthonormal basis of $\RR^n$. 

We define the Grassmannian $\Gr(k,n)$ as the space of (unoriented) $k$-dimensional linear subspaces of $\RR^n$; we identify $S\in\Gr(k,n)$ with the endomorphism $S\colon\RR^n\to\RR^n$ representing the orthogonal projection onto $S$. When no confusion shall arise and an orthonormal basis $\{\zeta_1,\dots,\zeta_k\}$ of $S$ is fixed, we identify $S$ with $\RR^k$ via the canonical bijection
\begin{equation*}
	\iota: S\ni x\mapsto(x\cdot\zeta_1,\dots,x\cdot\zeta_k)\in\RR^k:
\end{equation*}
therefore by $Sx$ we denote both the point $Sx\in S\subset\RR^n$ and its image via $\iota$.
In particular, when $S=\Span\{\e_1,\dots,\e_k\}$ and $x\in\RR^{n}$, we will often use the notation $x'=Sx = (x\cdot\e_1,\dots,x\cdot\e_k)\in\RR^k$.

We also let $S:\RR^{n,1}\to\RR^{k,1}$ be the map $S(x,t)=(Sx,t)$ and, in the case $S = \Span\{\e_1,\dots,\e_k\}$, for $X=(x,t)\in\RR^{n,1}$ we let $X' = (x',t)$.

Lastly, if $S$ and $T$ are two endomorphisms of $\RR^n$, we define the scalar product between $S$ and $T$ by
\begin{equation*}
	S\colon T = \sum_{i,j=1}^nS_{ij}T_{ij},
\end{equation*}
where $(S_{ij})$ is the representation of $S$ as a $n\times n$ matrix such that
\begin{equation*}
	S_{ij} = \e_i\cdot (S\e_j).
\end{equation*}
We also let $|S| = \sqrt{S\colon S}$.

\subsection{H\"{o}lder regularity}\label{sec:hold_reg}
\let \oldbeta \beta
\renewcommand{\beta}{\kappa}
We point out some facts and definitions on H\"{o}lder regularity for several objects. In what follows, $\beta\in(0,1)$ is a fixed parameter.
\begin{enumerate}
	\item \label{bullet:holder_space_functions}\textbf{Functions on $\RR^n$. }Given a function $u:\RR^n\supset U\to\RR^k$, we say that $u\in C^{1,\beta}(U;\RR^k)$ if there exists $C>0$ such that $\sup_{x\in U}|u(x)|\le C$ and for all $x\in U$ there is an affine function $L_x:\RR^n\to\RR^k$ such that, for every $y\in U$, it holds
	\begin{equation*}
		|u(y)-L_x(y)|\le C|x-y|^{1+\beta}.
	\end{equation*}
	
	\item\label{bullet:holder_ST_functions} \textbf{Functions on $\RR^{n,1}$. }Let $\Omega\subset\RR^{n,1}$. We say that $u:\Omega\to\RR^{k}$ is in $C^{1,\beta}(\Omega;\RR^k)$ if there exists $C>0$ such that $\sup_{X\in \Omega}|u(X)|\le C$ and for all $X\in \Omega$ there is an affine function $L_X:\RR^n\to\RR^k$ such that, for every $Y=(y,s)\in \Omega$, it holds
	\begin{equation*}
		|u(Y)-L_X(y)|\le C\rho(X,Y)^{1+\beta}.
	\end{equation*}
	
	\item \textbf{Submanifolds. }We say that a $k$-dimensional, properly embedded submanifold $\Gamma$ of some open set $U\subset\RR^n$ is $C^{1,\beta}$ if there exists some $\kappa>0$ such that, for every $x,y\in\Gamma$, it holds
	\begin{equation*}\label{eq:kappa}
	[\Gamma]_{C^{1,\beta}(U)}:= \sup_{\substack{x,y\in\Gamma\\x\neq y}}\frac{|T_x\Gamma-T_y\Gamma|}{|x-y|^{\beta}}<\infty,
	\end{equation*}
	where $T_\cdot\Gamma\in\Gr(k,n)$ is the tangent space to $\Gamma$ and $|T_x\Gamma-T_y\Gamma|$ should be intended as in \Cref{subsec:subspaces}.
\end{enumerate}

\begin{remark}
	We do not require the sets $U$ and $\Omega$ in items \ref{bullet:holder_space_functions} and \ref{bullet:holder_ST_functions} above to have any regularity.
	However, one can easily see that, if $U\subset\RR^{n}$ has  $C^1$ boundary and $u\in C^{1,\beta}(U;\RR^k)$, then $u$ is actually bounded in $U$, it is differentiable at every point of $\Int E$ and the usual definition of $C^{1,\beta}$ holds:
	\begin{equation*}
		||u||_{C^{1,\beta}(U)}:= \sup_U|u|+\sup_{\substack{x,y\in U\\x\neq y}}\frac{|\nabla u(x)-\nabla u(y)|}{|x-y|^\beta}<\infty.
	\end{equation*}
	In fact, $||u||_{C^{1,\beta}(U)}$ is bounded (up to some multiplicative constant depending only on $U, n,k,\beta$) by the same constant $C$ as in item \ref{bullet:holder_space_functions} above.
	
	Similarly, if $\Omega= U\times I\subset\RR^{n,1}$ for some $U\subset\RR^n$ with $C^1$ boundary and $I\subset\RR$ some interval, then $u\in C^{1,\beta}(\Omega;\RR^k)$ yields that $u$ is differentiable with respect to the space variable everywhere in $\Int U\times I$ and that
	\begin{equation*}
		\norm{u}_{C^{1,\beta}(\Omega)}:=\sup_\Omega |u| + \sup_{\substack{X,Y\in\Omega\\X\neq Y}}\frac{|\nabla u(X)-\nabla u(Y)|}{\rho(X,Y)^\beta} + \sup_{\substack{(x,t),(x,s)\in\Omega\\s\neq t}}\frac{|u(x,t)-u(x,s)|}{|t-s|^{(1+\beta)/2}}<\infty.
	\end{equation*}
\end{remark}

\renewcommand{\beta}{\oldbeta}

\subsection{Integral Varifolds}\label{subsec:varifolds}
We adopt most of the terminology from \cite{whitebdry}.
Let $U\subset\RR^d$ be an open set and let $\Mc(U)$ be the set of non-negative Radon measures on $U$; if $\phi$ is continuous and compactly supported on $U$, we let $M(\phi) = \int\phi(x)\,dM(x)$ for $M\in\Mc(U)$.
Let $\Mc_m(U)$ be the set of $m$-dimensional rectifiable non-negative Radon measures on $U$. Namely, $M\in\Mc_m(U)$ if and only if there exist a $m$-dimensional rectifiable set $E$ and a non-negative function $\theta\in L^1_{loc}(\Hc^m\rest E)$ such that
\begin{equation*}\label{eq:k_rect}
M(\phi)=\int_E\theta(x)\phi(x)\,d\haus^m(x)\qquad\mbox{for all }\phi\in C_c(U).
\end{equation*}
We also let $\IM_m(U)$ be the set of those $M\in \Mc_m(U)$ such that their density $\theta(x)$ is a non-negative integer at $M$-a.e. $x$.
If $M\in\Mc_m(U)$, then for $M$-a.e. $x$ the approximate tangent space $T_xM\in\Gr(m,d)$ is well defined (see, for instance, \cite[Chapter 3]{simonGMT}).
A $m$-dimensional varifold on $U$ is a Radon measure on $U\times\Gr(m,d)$ (see \cite[Chapter 8]{simonGMT}).
In particular, to each $M\in\Mc_m(U)$ we may associate a $m$-dimensional varifold $\Var(M)$ by
\begin{equation}\notag
\Var(M)(\phi)=\int\phi(x,T_xM)\,dM(x)\qquad\mbox{for all }\phi\in C_c(U\times \Gr(m,d)).
\end{equation}
Such an object is called a rectifiable varifold (see \cite[Chapter 4]{simonGMT}); $\Var(M)$ is said to be integral if and only if $M\in\IM_m(U)$.
If $M\in\Mc_m(U)$, we say that $\Var(M)$ has bounded first variation if there exists $C>0$ such that, for every smooth vector field $F:U\to\RR^d$ with compact support in $U$, it holds
\begin{equation*}
	\int T_xM\colon \nabla F(x)\,dM(x)\le C\norm{F}_\infty.
\end{equation*}
If $\Var(M)$ has bounded first variation, then there exist a $M$-locally integrable vector field $H_M$, a Radon measure $\beta_M$ that is singular with respect to $M$ and a $\beta_M$-locally integrable unit vector field $\zeta_M$ such that, for every $F\in C^1_c(U;\RR^d)$, it holds 
\begin{equation}\label{eq:first_var}
\int T_xM:\nabla F(x)\,dM(x)= -\int H_M\cdot F\,dM+\int F\cdot\zeta_M\,d\beta_M.
\end{equation}
In the following, we will often denote
\begin{equation}\notag
\dive_S F(x) = S\colon\nabla F(x).
\end{equation}
When $M\in\Mc_m(U)$, we also let
\begin{equation}\notag
\dive_MF(x) := \dive_{T_xM}F(x) = T_xM\colon \nabla F(x),
\end{equation}
whenever it is well defined.

\begin{definition}\label{def:vfd_w_b}
	Let $\Gamma$ be a properly embedded $(m-1)$-dimensional submanifold of $U\subset\RR^d$. We let $\Vc_m(U,\Gamma)$ be the space of those $M\in\IM_m(U)$ such that $\Var(M)$ has bounded first variation and the following hold true:
	\begin{enumerate}
		\item $\beta_M(E)\le\haus^{m-1}(E\cap \Gamma)$ for every $E\subset U$;
		\item $H_M(x)$ and $T_xM$ are perpendicular at $M$-a.e. $x$. \label{mean_curv_perp}
	\end{enumerate}
\end{definition}
As mentioned in the remark following \cite[Definition 6]{whitebdry}, \Cref{mean_curv_perp} is actually redundant, as it can be derived from \cite[\S5]{brakke}.

As in \cite{whitebdry}, for $M\in\Vc_m(U,\Gamma)$ we let
\begin{equation}\label{eq:nu}
\nu_M(x) = \lim_{r\searrow0}\frac{1}{\omega_{m-1}r^{m-1}}\int_{B_r(x)}\zeta_M\,d\beta_M
\end{equation}
where the limit exists, and $\nu_M(x)=0$ otherwise. Notice that the requirement $\beta_M\le\Hc^{m-1}\rest\Gamma$ in \Cref{def:vfd_w_b} yields $|\nu_M|\le 1$ $\haus^{m-1}\rest\Gamma$-a.e.. Moreover, by \cite[\S3.1]{allard1975}, $\nu_M(y)\perp\Gamma$ for $\Hc^{m-1}$-a.e. $y\in\Gamma$.

In the following, whenever $\Gamma$ is a $(m-1)$-dimensional submanifold of $\RR^d$, by a small abuse of notation we denote by $\Gamma$ the Hausdorff measure $\Hc^{m-1}\rest\Gamma$, if no confusion shall arise.


\subsection{Integral Brakke flows with boundary}\label{subsec:def_BF}
Let $U\subset \RR^d$ be an open set, $I\subset\RR$ be a non-empty interval and let $\Gamma$ be a properly embedded $(m-1)$-dimensional submanifold of $U$.
\begin{definition}[Integral Brakke flow]\label{def:bf}
	A $m$-dimensional integral Brakke flow with boundary $\Gamma$ in $U\times I$ is a collection $\Mb=\{M_t\colon t\in I\}\subset\Mc(U)$ such that the following hold true:
	\begin{enumerate}
		\item \label{a.e.rectif} for almost every $t$, $M_t\in\Vc_m(U,\Gamma)$;
		\item \label{assum:local_finit}if $I'\compact I$ and $U'\compact U$, then $\int_{I'}\int_{U'}(1+|H_{M_t}|^2)\,dM_t\,dt<+\infty$,
		\item if $[a,b]\subset I$ and $u$ is a non-negative, compactly supported, $C^1$ function on $U\times I$, then
		\begin{equation}\label{e:BF_def}
		\int u(\cdot,a)\,dM_a-\int u(\cdot,b)\,dM_b\ge\int_a^b\int\pt{u|H_{M_t}|^2-H_{M_t}\cdot\nabla u-\de_tu}\,dM_t\,dt.
		\end{equation}
	\end{enumerate}
	We denote by $\BF(U\times I,\Gamma)$ the set of all $m$-dimensional integral Brakke flows in $U\times I$ with boundary $\Gamma$.
\end{definition}
When $\Gamma=\emptyset$, we drop its indication and simply write $\BF(U\times I)$; notice that in this case $\beta_{M_t}=0$ for a.e. $t$, and the definition agrees with the one of integral Brakke flow (without boundary) given, for instance, in \cite{Tonegawabook}.

Given $\Mb\in\BF(U\times I,\Gamma)$, we define its \textit{space-time mass measure} $M$ by 
\begin{equation*}
\int\phi(x,t)\,dM(x,t) =\int\int\phi(x,t)\,dM_t\,dt
\end{equation*}
for every $\phi\in C_c(U\times I)$. We define the space-time track of $\Mb$ to be the closed set
\begin{equation*}\label{eq:st_track}
\Sigma_\Mb = \Clos\bigg(\bigcup_{t\in I }\supp M_t\times\{t\}\bigg),
\end{equation*}
where the closure is taken in the euclidean topology of $\RR^{d,1}$, and we let $\Sigma_\Mb(t)$ be the slice at time $t$ of $\Sigma_\Mb$, namely $\Sigma_\Mb(t) = \{x\in\RR^d\colon (x,t)\in\Sigma_\Mb\}$. 
It is straightforward to check that $\supp M\subset\Sigma_\Mb$. Under reasonable assumptions, the opposite inclusion holds true as well: we further discuss this point in \Cref{rk:supporti_coincidono}. Whenever no confusion may arise, we write $\Sigma$ and $\Sigma_t$ in place of $\Sigma_\Mb$ and $\Sigma_\Mb(t)$, respectively.

%
\begin{remark}[Scaling properties]\label{rk:scaling}
	A Brakke flow $\Mb\in\BF(U\times I,\Gamma)$ may be translated and parabolically dilated while preserving the requirements in \Cref{def:bf}. For $x_0\in\RR^d$ and $r>0$, let $T_{x_0,r}(y)=(y-x_0)/r$. By $(T_{x_0,r})\push\mu$ we denote the push-forward of $\mu\in\Mc(\RR^d)$ through $T_{x_0,r}$. Then $\Mb'=\{M'_s\}$ given by
	\begin{equation}\notag
	M_s'=r^{-m}(T_{x_0,r})\push M_{t_0+r^2s}
	\end{equation}
	is a Brakke flow in $\frac{U-x_0}{r}\times \frac{I-t_0}{r^2}$ with boundary $\frac{1}{r}(\Gamma-x_0)$. In this case, we will write
	\begin{equation}\notag
	\Mb' = \Dc_r(\Mb-X_0)
	\end{equation}
	where, as usual, $X_0=(x_0,t_0)$.
\end{remark}

\section{Properties of Integral Brakke flows with boundary}\label{sec:properties}
We collect some known results about Integral Brakke flows, which we will use throughout the rest of the paper. 

\subsection{Monotonicity properties}
We denote by $\hk:\RR^{d}\times(-\infty,0)\to\RR$ the $m$-dimensional backward heat kernel
\begin{equation*}\label{eq:hk_def}
\hk(x,t) = \frac{1}{(4\pi(-t))^{m/2}}\exp\bigg(-\frac{|x|^2}{4(-t)}\bigg).
\end{equation*}
We also pick a smooth cut-off function $\phin\in C^\infty_c([0,2))$ such that $\phin\equiv1$ in $[0,1]$, $|\phin'|\le2$ and $0\le\phin\le1$ everywhere, which from now on we consider fixed.
$\phin$ being chosen, for $R>0$ we set
\begin{equation}\notag
\hk_R(x,t) = \hk(x,t)\phin\bigg(\frac{|x|}{R}\bigg).
\end{equation}

\begin{proposition}[Huisken monotonicity formula]\label{prop:monoton}
	There exists a universal constant $C>0$ such that, if $\Mb\in\BF(U\times (-T,0),\Gamma)$ and $B_{2R}\subset U$, then for every $-T<s\le t<0$ it holds
	\begin{align}
	&\int\hk_R(x,t)\,dM_t-\int\hk_R(x,s)\,dM_s\notag\\
	&\qquad\le\int_{s}^{t}\int\nu_{M_\tau}\cdot\nabla\hk_R(\cdot,\tau)\,d\Gamma\,d\tau\label{integrale_brutto}\\
	&\qquad\quad+ C\frac{t-s}{R^2}\sup_{\tau\in[s,t]}\frac{M_\tau(B_{2R})}{R^m}\label{resto_bello},
	\end{align}
	where $\nu_{M_\tau}$ is defined in \eqref{eq:nu}.
\end{proposition}
\begin{proof}
	See \cite[Theorem 6.1]{whitebdry}. 
\end{proof}

In several points of the present work, we are going to need some precise bounds on \eqref{integrale_brutto} and \eqref{resto_bello}. While in most cases we will assume a uniform bound of the form
\begin{equation}\notag
\sup_t\sup_{B_r(x)}\frac{M_t(B_r(x))}{r^m}\le E_1<\infty
\end{equation}
which takes care of \eqref{resto_bello}, estimating \eqref{integrale_brutto} requires some more attention.
What we prove in the following lemma is that, at a small enough scale, \eqref{integrale_brutto} is close to $\frac{1}{2}$ if $0\notin\Gamma$, otherwise it is very small. 
\begin{lemma}\label{lemma:brutto}
	For every $\delta>0$, there exist small positive constants $\Lambda$ and $c$ with the following property. Let $U\subset\RR^d$ be open and let $\Gamma$ be a $C^{1,\alpha}$ submanifold of $U$. Then, for every $R\le c/[\Gamma]_{C^{1,\alpha}(U)}$ and for every $(x,t)\in U \times\RR$ such that $B_{2R}(x)\subset U$, it holds
	\begin{equation}\notag
	\int_{t-\Lambda R^2}^{t}\int|T_y\Gamma^\perp\nabla\hk_R(y-x,s-t)|\,d\Gamma(y)\,ds\le\frac{1}{2}\chi_{\Gamma^\compl}(x)+\delta.
	\end{equation}
\end{lemma}
The proof of \Cref{lemma:brutto} is somehow cumbersome and is therefore postponed to \Cref{sec:proofofLemmabrutto}.

Exploiting the above result, we may prove a sort of \textit{clearing-out lemma}, in the spirit of \cite[Proposition 3.6]{Tonegawabook}. Namely, we prove that, provided we have some control on \eqref{resto_bello} and \eqref{integrale_brutto}, if a point $(x,t)$ is in the space-time track of $\Mb$, then $M_s$ cannot be too small in a backward neighborhood of $(x,t)$.

Before proceeding with this result, we introduce the following terminology:
\begin{definition}[Maximal density ratio]
	A Brakke flow $\Mb$ (possibly with boundary) in $U\times I$ is said to have \emph{bounded maximal density ratio} in $U'\times I'$, where $U'\subset U$ and $I'\subset I$, if
	\begin{equation*}
	\mdr(\Mb,U'\times I'):=\sup_{B_r(x)\subset U'}\sup_{t\in I'}\frac{M_t(B_r(x))}{r^m}<\infty.
	\end{equation*}
\end{definition}

\begin{proposition}[Clearing-out lemma]\label{cout_lemma}
	For every $\bound<\infty$ there exist positive constants $c_1, c_2$ with the following property. Let $\Gamma$ be a $C^{1,\alpha}$ submanifold of $U$ and let $\Mb\in\BF(U\times (a,b),\Gamma)$ be such that
	\begin{equation*}\label{e:uniformbound}
	\mdr(\Mb,U\times(a,b))\le\bound.
	\end{equation*}
	If $(x,t)\in\Sigma_\Mb$, and $R$ is small enough depending on $\Gamma$, then 
	\begin{equation}\notag
	M_{t-c_1R^2}(B_{4R}(x))\ge c_2R^m.
	\end{equation}
\end{proposition}
\begin{proof}
	The proof of the case without a boundary can be found, for example, in \cite[Proposition 3.6]{Tonegawabook}. For the sake of completeness, we sketch the proof along the same line in the case of an Integral Brakke flow with boundary.
	
	Corresponding to $\delta=1/4$, choose $\Lambda$ and $c$ as in \Cref{lemma:brutto}. Let $(x,t)\in\Sigma$ and let $R\le c/[\Gamma]_{C^{1,\alpha}(U)}$.
	
	We first assume that $x\in\supp M_t$ and that $M_t\in\Vc_m(U,\Gamma)$, so that in particular $M_t=\theta(\cdot)\haus^m\rest E$ for some $m$-rectifiable set $E$. Then there exists $y\in B_{R}(x)$ such that
	\begin{equation}\label{eq:density_integer}
	1\le\theta(y)=\lim_{\tau\nearrow0}M_{t}(\hk_R(\cdot-y,\tau)).
	\end{equation}
	Therefore, by centering \Cref{prop:monoton} at a point $(y,t-\tau)$ and then letting $\tau\nearrow0$, for any $t_1<t$, it holds
	\begin{equation}\notag
	M_{t_1}(\hk_R(\cdot-y,t_1-t))\ge\theta(y)-C\bound\frac{t-t_1}{R^2}-\int_{t_1}^t\int\nu_{M_s}\cdot\nabla\hk_R(\cdot-y,s-t)\,d\Gamma\,ds.
	\end{equation}
	We now choose $c_1$ so small that both $C\bound c_1\le\frac{1}{8}$ and $c_1\le\Lambda$ and we set $t_1=t-c_1R^2$. Then, using \Cref{lemma:brutto}, we obtain
	\begin{equation}\notag
	M_{t_1}(\hk_R(\cdot-y,-c_1R^2))\ge\theta(y)-\frac{1}{8}-\bigg(\frac{1}{2}+\frac{1}{4}\bigg)\ge\frac{1}{8},
	\end{equation}
	where the second inequality is given by \eqref{eq:density_integer}.
	Notice that, for every $z\in\RR^d$, simple computations yield
	\begin{equation*}
		\hk_R(z-y,-c_1R^2)\le CR^{-m}\chi_{B_{2R}(y)}(z)\le CR^{-m}\chi_{B_{3R}(x)}(z)
	\end{equation*}
	for some $C>0$ universal. Hence, by integrating the above inequality in $M_{t-c_1R^2}$, we obtain
	\begin{equation}\notag
	M_{t-c_1R^2}(B_{3R}(x))\ge \frac{R^m}{C}M_{t-c_1R^2}(\hk_R(\cdot-y,-c_1R^2))\ge \frac{R^m}{8C},
	\end{equation}
	as desired.
	
	If $x\notin\supp M_t$ or $M_t\notin\Vc_m(U,\Gamma)$, then one can find a sequence of points $(x_i,t_i)$ such that $M_{t_i}\in\Vc_m(U,\Gamma)$, $x_i\in\supp M_{t_i}$ and such that $(x_i,t_i)\to(x,t)$. It is then sufficient to choose $R_i$ so that $t_i-c_1R_i^2=t-c_1R^2$ to obtain, for $i$ large enough,
	\begin{equation}\notag
	M_{t-c_1R^2}(B_{4R}(x))\ge M_{t-c_1R^2}(B_{3R}(x_i))\ge c_2 R^m.
	\end{equation}
\end{proof}

We now state two important consequences of \Cref{cout_lemma}.
\begin{lemma}\label{rk:supporti_coincidono}
	Let $\Mb\in\BF(U\times I,\Gamma)$ have bounded maximal density ratio in $U\times I$ and let $\Gamma\in C^{1,\alpha}(U)$.
	Then 
	\begin{equation}\notag
	\Sigma_\Mb=\supp M.
	\end{equation}
\end{lemma}
\begin{proof}
	The inclusion $\supp M\subset\Sigma$ is trivial. For the opposite inclusion, notice that, for a point $(x,t)\in\Sigma$ and for every $r>0$ small enough, \Cref{cout_lemma} gives
	\begin{equation}\notag
	M_{t-cr^2}(B_{r}(x))\ge cr^m.
	\end{equation}
	It is now sufficient to integrate this inequality in $r$ to obtain that for every $r>0$ small enough, there is a  set of the form
	\begin{equation}\notag
	A_r = \big\{(y,s)\colon|y-x|\le\theta\sqrt{t-s}\le r\big\}
	\end{equation}
	for some positive $\theta,c$ depending only on $\mdr(\Mb)$ such that $M(A_r)>0$, hence $(x,t)\in\supp M$, as claimed.
\end{proof}

\begin{lemma}\label{lemma:federer}
	Let $\Mb\in\BF(U\times I,\Gamma)$ have bounded maximal density ratio and let $\Gamma\in C^{1,\alpha}(U)$. Then
	\begin{equation}\notag
	M\ge\haus^{m,1}\rest\Sigma_{\Mb}.\label{federer_tesi}
	\end{equation}
\end{lemma}
For the proof of the above lemma, we refer the reader to \Cref{app:proof_of_federer}.

\subsection{Maximum principle}\label{subsec:MAX}
In the present subsection, we assume that $U\subset\RR^d$ is open and $I\subset\RR$ is an interval of the form $\Tint{a,b}$. We also let $\Gamma$ be a $(m-1)$-dimensional, $C^{1,\alpha}$ submanifold of $U$.

The main result of the present section is the following maximum principle.

\begin{proposition}[Maximum principle]\label{Maxprinciple}
	Let $\Mb\in\BF(U\times I,\Gamma)$ have bounded maximal density ratio.
	
	If there exist $u\in C^2(U\times I)$ and a point $(x_0,t_0)\in \Sigma\setminus \de_p(U\times I)$ with $x_0\notin\Gamma$ such that $u|_{\Sigma\cap\{t\le t_0\}}$ has a local maximum at $(x_0,t_0)$ and $\nabla u(x_0,t_0)\neq0$, then
	\begin{equation}\notag
	\de_tu(x_0,t_0)-\inf_{\substack{T\in\Gr(m,d)\\T\perp\nabla u(x_0,t_0)}} T\sprod D^2u(x_0,t_0)\ge0.
	\end{equation}
\end{proposition}
\begin{proof}
	\newcommand{\backball}{\pball^-}
	This proposition is a corollary of the results in \cite[Section 13]{White:avoidance}, see also \cite{AmbrosioSoner}; for the reader's convenience, we give a self-contained proof, in the spirit of, for example, \cite{white_areablowup}.
	
	We may assume, without loss of generality, that $(x_0,t_0)=(0,0)$ and that $u|_{\Sigma\cap\{t\le0\}}$ has a strict local maximum at $(0,0)$ (if not, replace $u$ by $u-|x|^4-|t|^2$).
	
	\textbf{Step 1. }We first prove that
	\begin{equation}\notag
	\de_tu(0,0)-\trace_m D^2u(0,0)\ge0,
	\end{equation}
	where $\trace_mD^2u$ is the sum of the $m$ smallest eigenvalues of $D^2u$. Assume the result does not hold. Arguing as in \cite[Lemma 2.4]{white_areablowup}, we may assume that, for some $\rho>0$ and $\eps>0$ small, $u$ satisfies:
	\begin{enumerate}[(i)]
		\item \label{point:first} $\de_tu-\trace_mD^2u<-\eps<0$ in $Q_\rho$;
		\item $B_\rho\cap\Gamma=\emptyset$;
		\item \label{point:last} $u>\eps>0$ in $\Sigma\cap Q_{\rho/2}$ and $u<0$ in $\Sigma\cap\{t\le0\}\setminus Q_\rho$;
	\end{enumerate}
	We now let $\phi(x,t) = (u^+(x,t))^4$, where $u^+=\max\{u,0\}$, and we use $\phi$ as a test function for \eqref{e:BF_def}. Since $\phi(\cdot, -\rho^2)=0$ by assumption, we have
	\begin{align*}
	0&\le\int\phi(\cdot,0)\,dM_{0}\\
	&=\int\phi(\cdot,0)\,dM_{0} - \int\phi(\cdot,-\rho^2)\,dM_{-\rho^2}\\
	&\le\int_{-\rho^2}^0\int\bigg(-|H|^2\phi+H\cdot\nabla\phi + \de_t\phi  \bigg)\,dM_t\,dt
	\end{align*}
	where the last inequality is given by \eqref{e:BF_def} and we have set $H(\cdot,t) = H_{M_t}(\cdot)$ for a.e. $t$. We now use the fact that $\supp\phi\subset\Gamma^\compl$, thus
	\begin{equation}\notag
	\int H\cdot\nabla\phi\,dM_t = -\int\dive_{M_t}\nabla\phi\,dM_t
	\end{equation}
	for a.e. $t$. Since the term $|H|^2\phi$ is non-negative, we obtain from the above chain of inequalities:
	\begin{equation}\notag
	0\le\int_{-\rho^2}^0\int\bigg(-\dive_{M_t}\nabla\phi+\de_t\phi  \bigg)\,dM_t\,dt.
	\end{equation}
	Some straightforward computations show that
	\begin{equation}\notag
	\dive_{M_t}\nabla\phi = 4(u^+)^3\dive_{M_t}\nabla u\ge4(u^+)^3\trace_mD^2u
	\end{equation}
	and $\de_t\phi= 4(u^+)^3\de_tu$. Therefore 
	\begin{align*}
	0\le\int_{-\rho^2}^0\int4(u^+)^3\big(\de_tu-\trace_mD^2u\big)\,dM_t\,dt\le-4\eps^4M(Q_{\rho/2})
	\end{align*}
	where the last inequality is given by \Cref{point:first} and \Cref{point:last} above. In particular, it must be
	\begin{equation}\notag
	M(Q_{\rho/2})=0;
	\end{equation}
	however, by \Cref{rk:supporti_coincidono}, $(0,0)\in\Sigma=\supp M$, thus we reach a contradiction.
	
	\textbf{Step 2. }We now prove the general result. It is sufficient to show that one can find a $m$-dimensional subspace $\bar T$ of $\RR^d$ such that
	\begin{equation}\notag
	\de_tu (0,0)-\bar T\sprod D^2u(0,0)\ge0
	\end{equation}
	and $\bar T\nabla u(0,0)=0$.
	Without loss of generality, assume that $u(0,0)=0$. 
	Let $\psi_j(z) = z+\frac{j}{2}z^2$
	and let
	\begin{equation}\notag
	u_j(X) = \psi_j(u(X)).
	\end{equation}
	Then, for every $j$, $u_j|_{\Sigma\cap\{t\le0\}}$ has a local maximum at $(0,0)$. Therefore, by Step 1, there is a $m$-dimensional subspace $T_j$ of $\RR^d$ such that, at $(0,0)$,
	\begin{equation}\notag
	0\le \de_tu_j-T_j\sprod D^2u_j=\de_tu-T_j\sprod D^2u-jT_j\sprod(\nabla u\otimes\nabla u).
	\end{equation}
	Up to a subsequence, which we do not relabel, we have that $T_j\to \bar T$ for some $m$-dimensional subspace $\bar T$, and
	\begin{equation}\notag
	\bar T\sprod(\nabla u\otimes\nabla u)\le\liminf_j\frac{1}{j}(\de_tu-T_j\sprod D^2u)=0,
	\end{equation}
	thus $\bar T\perp \nabla u$. On the other hand, since $jT_j\colon(\nabla u\otimes\nabla u)\ge0$, we have
	\begin{equation}\notag
	\bar T\sprod D^2u \le\liminf_jT_j\sprod D^2u\le\liminf_j(\de_tu-j T_j\sprod(\nabla u\otimes\nabla u))\le\de_tu,
	\end{equation}
	as desired.
\end{proof}

Given a upper-semicontinuous function $u:\RR^{m,1}\to[0,1]\cup\{-\infty\}$ and a smooth function $\phi:\RR^{m,1}\to\RR$, we say that $\phi$ \textit{touches $u$ from above at $(x_0',t_0)\in \RR^{m,1}$} if there exists $r>0$ such that
\begin{equation}\notag
\begin{cases}
\phi(x',t)\ge u(x',t)\quad\mbox{for every }(x',t)\in Q_r^m(x_0',t_0),\\
\phi(x_0',t_0)=u(x'_0,t_0).
\end{cases}
\end{equation}

We recall the definition of Pucci's maximal operator (see, for instance, \cite[Section 2.2]{CC}). For a symmetric matrix $N\in\RR^{d\times d}$, we let
\begin{equation}
\Mc^+(N):=\Mc^+\Big(N,\frac{1}{2},2\Big)=\frac{1}{2}\sum_{\lambda_i<0}\lambda_i + 2\sum_{\lambda_i>0}\lambda_i,
\end{equation}
where $\lambda_i = \lambda_i(N)$ are the eigenvalues of $N$.
The following result is a consequence of \Cref{Maxprinciple}.

\begin{corollary}\label{cor:max_princ2}
	Let $\Mb\in\BF(\RR^{d,1})$ have bounded maximal density ratio. For every $(x',t)\in \RR^{m,1}$, let
	\begin{equation}\notag
	u(x',t) = \sup\{|z|\colon z\in\RR^{d-m}\mbox{ and }(x',z)\in\Sigma_{\Mb}(t)\}
	\end{equation}
	with the convention $\sup\emptyset=-\infty$ and assume that $u\le1$ everywhere. There is $\delta>0$ universal such that, whenever a smooth function $\phi:\RR^{m,1}\to\RR$ touches $u$ from above at $X_0'=(x_0',t_0)$ and $\max\{|D^2\phi(X_0')|,|\nabla\phi(X_0')|\}\le \delta$, it holds
	\begin{equation*}\label{eq:tesimax2}
	\de_t \phi (X_0')- \Mc^+(D^2\phi(X_0'))\le0.
	\end{equation*}
\end{corollary}
\let \oldeta \eta
\renewcommand{\eta}{\zeta}
\begin{proof}
	We assume $x_0'=0$ and $t_0=0$.
	Notice that, since $\Sigma$ is closed and $u(0,0)=\phi(0,0)$, the supremum in the definition of $u$ is attained and, without loss of generality, we may assume that the contact point is $x_0=\phi(0,0) \e_{d}\in\Sigma_{0}$. 
	We let $S = \Span\{\e_1,\dots,\e_m\}$ and $S' = \Span\{\e_{m+1},\dots,\e_{d-1}\}$, so that $\RR^d = S+S'+\Span\{\e_d\}$.
	Consider the function
	\begin{equation*}
	H(x,t) = \frac{1}{4}\big|S' x\big|^2+x\cdot\e_d - \phi(Sx,t).
	\end{equation*}
	By assumption, in a neighborhood of $(x_0,0)$ it holds
	\begin{equation*}
	|S^\perp x|\le \phi(Sx,t)\le2
	\end{equation*}
	for every $(x,t)\in\Sigma$; therefore it can be checked that $H|_{\Sigma\cap\{t\le0\}}\le0$ in the same neighborhood. Since $H(x_0,0)=0$, $H|_{\Sigma\cap\{t\le0\}}$ has a local maximum at $(x_0,0)$. Hence, by \Cref{Maxprinciple}, it holds
	\begin{equation}\label{eq:max_00}
	\de_tH(x_0,0)-\inf_{T\perp\nabla H(x_0,0)} T\colon D^2H(x_0,0)\ge0.
	\end{equation}
	We now estimate the two summands in the above inequality. In order to do so, we first remark that
	\begin{equation*}
	\nabla H(x_0,0) = \begin{pmatrix}
	-\nabla\phi(0,0)\\0\\1
	\end{pmatrix}
	\qquad
	D^2H(x_0,0) = \begin{pmatrix}
	-D^2\phi(0,0)&0&0\\0&\frac{1}{2}\Id_{S'}&0\\0&0&0
	\end{pmatrix}.
	\end{equation*}
	Consider $T\in\Gr(m,d)$ and an orthonormal basis $\eta_1,\dots,\eta_m$ of $T$. Then
	\begin{align*}
	T\colon D^2H
	&= \sum_{i=1}^m\iprod{D^2H\eta_i}{\eta_i}\\
	&=\sum_{i=1}^m\bigg(-\iprod{D^2\phi(Sx,t)S\eta_i}{S\eta_i} + \frac{1}{2}\big|S'\eta_i\big|^2\bigg)\\
	&\ge-\sum_{i=1}^m\iprod{D^2\phi(Sx,t)S\eta_i}{S\eta_i}.
	\end{align*}
	In particular, $S\colon D^2H(x,t) = -\Delta\phi(Sx,t)$ and
	\begin{align*}
	T\colon D^2H &= S\colon D^2H + (T-S)\colon D^2H\\
	&\ge -\Delta\phi - |T-S||D^2\phi|,
	\end{align*}
	Now, if $|T-S|\le c_1$, then the above inequality yields that, for some small $c_1$ universal,
	\begin{equation*}
	T\colon D^2H(x_0,0)\ge - \Mc^+(D^2\phi(0,0)).
	\end{equation*}
	On the other hand, if $|T-S|\ge c_1$, then we may choose an orthonormal basis $\eta_1,\dots,\eta_m$ of $T$ such that $|S^\perp\eta_1|\ge c_2$ for some $c_2$ universal. Since we are also assuming $T\perp\nabla H(x_0,0)$, we have
	\begin{equation*}
	0 = \eta_1\cdot \nabla H(x_0,0) = -S\eta_1\cdot\nabla\phi(0,0) + \eta_1\cdot \e_d
	\end{equation*}
	thus, in particular, $|\eta_1\cdot\e_d|\le|\nabla\phi|\le\delta$ and
	\begin{equation*}
	|S'\eta_1| \ge |S^\perp\eta_1|-|\eta_1\cdot \e_d|\ge c_2-\delta\ge\frac{c_2}{2},
	\end{equation*}
	provided $\delta\le c_2/2$. 
	Therefore
	\begin{align*}
	T\colon D^2H(x_0,0)\ge-\sum_{i=1}^m\iprod{D^2\phi(0,0)\eta_i}{\eta_i} + \frac{1}{2}|S'\eta_1|^2\ge - \Delta\phi(0,0)+\frac{c_2^2}{8}\ge-C\delta + \frac{c_2^2}{8}
	\end{align*}
	for some $C$ universal, since $|D^2\phi(0,0)|\le\delta$ by assumption.
	We may choose $\delta$ smaller, if needed, so that
	\begin{equation*}
	- \Mc^+(D^2\phi(0,0))\le C\delta\le -C\delta + \frac{c_2^2}{8}.
	\end{equation*}
	Therefore, whether $|T-S|\le c_1$ or not, it holds
	\begin{equation*}
		T\colon D^2H(x_0,0)\ge-\Mc^+(D^2\phi(0,0)).
	\end{equation*}
	
	We conclude the proof by remarking that
	\begin{equation*}
	\de_t H(x_0,0) = -\de_t\phi(0,0),
	\end{equation*}
	thus \eqref{eq:max_00} gives the desired result.
	
\end{proof}
\renewcommand{\eta}{\oldeta}

\begin{remark}
	With some more accurate computations, one may show that, actually, at the contact point $\phi$ satisfies the following inequality:
	\begin{equation*}
		\de_t \phi - \sqrt{1+|\nabla\phi|^2}\dive\bigg(\frac{\nabla\phi}{\sqrt{1+|\nabla\phi|^2}}\bigg)\le0.	
	\end{equation*}
	However, the weaker result proved in \Cref{cor:max_princ2} will be sufficient for the rest of the paper.
\end{remark}

\section{Improvement of flatness}\label{sec:IOF}
This section is the core of the present work. We prove that if a Brakke flow with boundary is sufficiently flat in $Q_1$, then its flatness can be improved at a smaller universal scale. This is going to allow us to prove the desired $C^{1,\beta}$ regularity: see \Cref{sec:final}.

We introduce the following notation. 
We fix a $m$-dimensional subspace of $\RR^d$, which we denote by $S$, and a $(m-1)$-dimensional subspace of $\RR^d$, which we denote by $\Gamma_0$, such that $\Gamma_0\subset S$. Up to changing coordinates in $\RR^d$, we shall assume for the rest of the present section that $S=\Span\{\e_1,\dots,\e_{m}\}$ and that $\Gamma_0=\Span\{\e_1,\dots,\e_{m-1}\}$. We also let $S^+=S\cap\{x_{m}>0\}$.

Given a $(m-1)$-dimensional submanifold $\Gamma$ of $B_R$, we write $\Gamma\in\Flat_\alpha(\delta,B_R)$ if the following hold true:
\begin{itemize}
	\item $\Gamma$ is a $C^{1,\alpha}$ submanifold of $B_R$ and $[\Gamma]_{C^{1,\alpha}(B_R)}\le\delta R^{-\alpha}$.
	\item $0\in\Gamma$ and $T_0\Gamma=\Gamma_0$.
\end{itemize}

In passing, we remark that if $\Gamma\in\Flat_\alpha(\delta,B_R)$ and $\theta>0$, then $\theta\Gamma\in\Flat_\alpha(\delta,B_{\theta R})$.

Moreover, if $\Gamma\in\Flat_\alpha(\delta,B_R)$, and $\delta$ is smaller than some constant depending only on $\alpha$, then there exists $\gamma:\Gamma_0\cap B_R\to\Gamma_0^\perp$ such that $|\gamma(0)|=|\nabla\gamma(0)|=0$, $||\gamma||_{C^{1,\alpha}(B_R)}\le\delta R^{-\alpha}$ and
\begin{equation*}\label{eq:gamma_graph}
\Gamma=\{x+\gamma(x)\colon x\in \Gamma_0\cap B_R\}\cap B_R;
\end{equation*}
given $\Gamma\in\Flat_\alpha(\delta,B_R)$, we will always implicitly define $\gamma$ as above.

The following is the main result of the present section.

\begin{theorem}[Improvement of flatness]\label{thm:IOF}
	For every $E_0$ and $\alpha$, there exist constants $\Lambda,\eps_0, \eta, \beta$ (small) and $C$ (large) with the following property. Let $\eps\le\eps_0$, $\Gamma\in\Flat_\alpha(\eps,B_1)$ and $\Mb\in\BF(B_1\times[-\Lambda,0],\Gamma)$ be such that $(0,0)\in\Sigma_{\Mb}$,
	\begin{equation*}\label{eq:hp_iof_flat}
	\Sigma_\Mb\subset\{(z,\tau)\colon \dist(z,S^+)\le\eps\},
	\end{equation*}
	\begin{equation}\notag
	\sup_{t\in[-\Lambda,0]}M_t(B_1)\le E_0
	\end{equation}
	and
	\begin{equation}\label{eq:double_dens_0}
	\int_{B_1}\hk(\cdot,-\Lambda)\,dM_{-\Lambda}\le\frac{3}{4}.
	\end{equation}	
	Then there exists a half plane $T^+$ of the form
	\begin{equation}\label{eq:half_plane}
	T^+ = \{x+w\zeta:x\in\Gamma_0,w>0\}
	\end{equation}
	for some $\zeta\in\Gamma_0^\perp$ with $|\zeta-\e_{m}|\le C\eps$, such that
	\begin{equation}\label{eq:IOF_conclusion}
	\Sigma_\Mb\cap Q_\eta\subset\bigg\{(x,t)\colon \dist(x,T^+)\le\eta^{1+\beta}\eps\bigg\}.
	\end{equation}
\end{theorem}

The proof of \Cref{thm:IOF} is based on a contradiction and compactness argument. If one assumes the conclusion does not hold, then it is possible to find a sequence of Brakke flows which are flatter and flatter and satisfy the other assumptions of \Cref{thm:IOF}, for which, however, no half-plane of the form \eqref{eq:half_plane} can be found so that the flatness improves at any smaller scale. However, for such flows, one shows that, after an appropriate rescaling, the space-time tracks must converge in the Hausdorff distance to the graph of a solution to the heat equation. It is then sufficient to use Schauder estimates for the heat equation with Dirichlet boundary condition to recover the conclusion.

The central point of the proof is to obtain the desired compactness.
This is achieved via the two following results.
The first one provides a control over the oscillations near $\Gamma$ of the space-time support of a Brakke flow satisfying the assumptions of \Cref{thm:IOF}.
\begin{proposition}[Boundary behaviour]\label{prop:barrierSTATEMENT}
	For every $E_0$ and $\alpha$, there exist small constants $c_1$ and $r_1$ with the following property. Let $\Mb$ and $\Gamma$ satisfy the assumptions of \Cref{thm:IOF}.  Then
	\begin{equation}\notag
	\Sigma\cap Q_{r_1}\subset\bigg\{(x-\gamma(x''))\cdot \e_{m}\ge -\eps^2+ c_1\frac{|S^\perp(x-\gamma(x''))|^2}{2\eps^2}\bigg\}.
	\end{equation}
\end{proposition}
Here, $x''$ denotes the point $(x_1,\dots,x_{m-1},0,\dots,0)\in\Gamma_0$.

With the above results at hand, we may prove that, if the Brakke flow is flat enough, then assumption \eqref{eq:double_dens_0} gives a Holder-type modulus of continuity in parabolic cylinders whose radii are controlled from below by some power of the flatness $\eps$.

\begin{proposition}[Decay of oscillations]\label{prop:holder_STATEMENT}
	For every $E_0$ and $\alpha$, there exist constants $\varsigma, C_2$ and $r_2$ with the following property. Let $\Mb$ and $\Gamma$ satisfy the assumptions of \Cref{thm:IOF} and let $(x,t),(y,s)\in\Sigma\cap Q_{r_2}$. If $\min\{x_{m},y_{m}\}\ge 2\eps$ and
	\begin{equation}\notag
	\rho:=\rho((x',t),(y',s))\ge C_2\eps^\varsigma,
	\end{equation}
	then
	\begin{equation}\notag
	|S^\perp(x-y)|\le C_2\eps \rho^\varsigma.
	\end{equation}
\end{proposition}

The two above results are sufficient to prove, via an Arzelà-Ascoli-type argument, the convergence in the Hausdorff distance which we have described. 

Before proceeding, it is worth spending a few words on how the constants in \Cref{thm:IOF} will be chosen.
\begin{itemize}
	\item We fix $\Lambda$ once and for all in \Cref{prop:propagation}; it will be needed to prove that $\Mb$ has bounded maximal density ratio in a smaller parabolic cylinder, $Q_{r_3}$.
	\item Propositions \ref{prop:barrierSTATEMENT}, \ref{prop:holder_STATEMENT} and \ref{prop:propagation} hold true provided $\eps_0$ is small enough (depending on $E_0$). We will therefore always assume that this is the case. The final value of $\eps_0$ will not be determined explicitly, as \Cref{thm:IOF} is proved by compactness.
	\item The constants $r_1$ and $r_2$ chosen in Propositions \ref{prop:barrierSTATEMENT} and \ref{prop:holder_STATEMENT} are chosen smaller than $r_3$ (determined in \Cref{prop:propagation}) and they depend on $E_0$ and $\alpha$. These two constants will give upper bounds for $\eta$. We will then give a further upper bound for $\eta$ coming from the regularity properties of the heat equation.
	\item Lastly, the constants $C$ and $\beta$ depend only on $\alpha$ and on regularity properties for the heat equation.
\end{itemize}

We now briefly describe the rest of the present section.
The proof of \Cref{thm:IOF} is given in \Cref{subsec:iof_proof}. In \Cref{subsec:iof_prelim}, we state and prove some lemmas which will be useful in the following.
The proofs of \Cref{prop:barrierSTATEMENT} and \Cref{prop:holder_STATEMENT} are postponed to \Cref{subsec:bdry} and \Cref{subsec:osc_decay}, respectively.

\subsection{Preliminaries to the proof of \texorpdfstring{\Cref{thm:IOF}}{the improvement of flatness}}\label{subsec:iof_prelim}
Some remarks on the assumptions of \Cref{thm:IOF} will be needed for the proofs of \Cref{prop:barrierSTATEMENT}, \Cref{prop:holder_STATEMENT} and, ultimately, of \Cref{thm:IOF} itself.
We begin by showing that \eqref{eq:double_dens_0} propagates in the interior of the domain. 
\begin{proposition}[Propagation of small density]\label{prop:propagation}
	For every $E_0$ and $\alpha$, there is $r_3$ small with the following property. Let $\Mb$ and $\Gamma$ satisfy the assumptions of \Cref{thm:IOF}. Then, for every $(x,t)\in Q_{r_3}$ and for every $\tau\in(-r_3^2,0)$, it holds
	\begin{equation}\notag
	\int_{B_{r_3}(x)}\hk(\cdot-x,\tau)\,dM_t\le\frac{7}{8}+\frac{1}{2}\chi_{\Gamma^\compl}(x).
	\end{equation}
\end{proposition}
\begin{proof}	
	We fix positive constants $r_3\le\frac{1}{8}$, $\eps$, $\Lambda$ and $\delta$, all of which we will determine later; we always assume that $r_3$ is much smaller than $\Lambda$.
	For simplicity of notation, in this proof we set $r=r_3$.
	For $(x,t)\in Q_r$ and $\tau\in(-r^2,0)$, we let $t_0=t-\tau$. Then, by \Cref{prop:monoton}, it holds
	\begin{align}
	&\int_{B_r(x)}\hk(\cdot-x,t-t_0)\,dM_t\le\int\hk_{1/8}(\cdot-x,t-t_0)\,dM_t\notag\\
	&\qquad\le\int\hk_{1/8}(\cdot-x,-\Lambda-t_0)\,dM_{-\Lambda}\label{eq:boundamiquesta}\\
	&\qquad\quad+\int_{-\Lambda}^t\int\nu_M\cdot\nabla\hk_{1/8}(\cdot-x,s-t_0)\,d\Gamma\,d\tau+C E_0 (t+\Lambda)\notag.
	\end{align}
	
	By \Cref{lemma:brutto}, if $\eps$ and $\Lambda$ are small enough and $r$ is much smaller than $\Lambda$, then
	\begin{align*}
	&\int_{-\Lambda}^{t}\int\nu_M\cdot\nabla\hk_{1/8}(\cdot-x,s-t_0)\,d\Gamma\,d\tau\\
	&\qquad\le\int_{t_0-2\Lambda}^{t_0}\int|T_y\Gamma^\perp\nabla\hk_{1/8}(y-x,s-t_0)|\,d\Gamma(y)\,d\tau\\
	&\qquad\le\frac{1}{2}\chi_{\Gamma^\compl}(x)+\delta.
	\end{align*} 
	We then take $\Lambda$ even smaller so that $CE_0(t+\Lambda)\le CE_0\Lambda\le \delta$.
	
	So far, we have fixed $\eps$ and $\Lambda$ depending only on $E_0$ and $\delta$, and we have assumed that $r$ is much smaller than $\Lambda$. The last step is to choose $r$ even smaller in order to bound \eqref{eq:boundamiquesta} from above. To this end, we let $L$ be the Lipschitz constant of $\hk$ restricted to $\RR^d\times(-\infty,-\Lambda/2]$. Since $r$ is much smaller than $\Lambda$, then $-\Lambda-t_0\le-\Lambda/2$ and we can estimate, for every $y\in B_{1/4}(x)$,
	\begin{align*}
	\hk_{1/8}(y-x,-\Lambda-t_0)&\le\hk(y-x,-\Lambda-t_0)\\
	&\le \hk(y,-\Lambda) + L(|x|+|t_0|)\\
	&\le \hk(y,-\Lambda)+2Lr.
	\end{align*}
	Let now $b = b(\Lambda)>0$ be so small that $\hk(y,-\Lambda)\ge b$ if $|y|\le1/2$. In particular, assuming that $r\le1/4$, for every $y\in B_{1/4}(x)$, it holds
	\begin{equation}\notag
	\hk_{1/8}(y-x,-\Lambda-t_0)\le \bigg(1+\frac{2Lr}{b}\bigg)\hk(y,-\Lambda)
	\end{equation}
	The same bound holds, trivially, for any $y$ such that $|y-x|\ge1/4$. We now choose $r$ even smaller, if needed, so that $\frac{2Lr}{b}\le \delta$. Therefore we may bound
	\begin{equation}\notag
	\int\hk_{1/8}(\cdot-x,t_0+\Lambda)\,dM_{-\Lambda}
	\le (1+\delta)\int_{B_1}\hk(\cdot,-\Lambda)\,dM_{-\Lambda}\le\frac{3}{4}(1+\delta),
	\end{equation}
	which yields the desired conclusion, up to choosing $\delta$ small universal.
\end{proof}

\begin{corollary}[Bound on $\mdr(\Mb)$]\label{cor:MDR}
	Under the assumptions of \Cref{prop:propagation},
	there exist $E_1$ universal such that, for every $t\in[-r_3^2,0]$ and every $B_r(x)\subset B_{r_3}$, it holds
	\begin{equation}\label{eq:MDR_0}
	M_t(B_r(x))\le E_1 r^m.
	\end{equation}
	In particular, for every $(x,t)\in\Sigma_\Mb\cap Q_{r_3}$ and for every $r>0$ small enough, it holds
	\begin{equation}\label{eq:cout_equation}
	M_{t-c_1r^2}(B_r(x))\ge c_2 r^m
	\end{equation}
	for some $c_1,c_2$ small universal.
\end{corollary}
\begin{proof}
	Let $x,t$ and $r$ as in the statement. Then
	\begin{equation}\notag
	M_t(B_r(x))\le C r^{m}\int_{B_r(x)}\hk(\cdot-x,-r^2)\,dM_t\le 2 C r^m.
	\end{equation}
	\eqref{eq:cout_equation} follows from \eqref{eq:MDR_0} and \Cref{cout_lemma}.
\end{proof}

\subsection{Proof of \texorpdfstring{\Cref{thm:IOF}}{the improvement of flatness}}\label{subsec:iof_proof}
As stated earlier, we are going to argue by contradiction and compactness. Namely, we fix $E_0$ and $\alpha$, we let $\Lambda$ be as specified in \Cref{prop:propagation} and we assume there exist $\eps_j\searrow0$ and two sequences $\{\Gamma^j\}$, $\{\Mb^j\}$ such that, for every $j$, $\Mb^j$ and $\Gamma^j$ satisfy the assumptions of \Cref{thm:IOF} with $\eps_0$ replaced by $\eps_j$.

In particular, we assume that $\Gamma^j\in\Flat_\alpha(\eps_j,B_1)$ and
\begin{equation}
\Sigma_{\Mb^j}\subset\{(z,\tau)\colon\dist(z,S^+)\le\eps_j\}.\label{eq:flatness_ripetuta}
\end{equation}

We also assume, for the sake of contradiction, that for no $j$ \eqref{eq:IOF_conclusion} is satisfied for any choice of $T^+$, $\eta$ and $\beta$.

In the following, we let $\gamma^j:\Gamma_0\cap B_1\to\Gamma_0^\perp$ be such that $\Gamma^j\cap B_1\subset\graph\gamma^j$, as in the definition of $\Flat_\alpha(\eps_j,B_1)$, and we let $\Sigma^j:=\Sigma_{\Mb^j}$.
We also fix $r_0 = \min\{r_1,r_2,r_3\}$, so that the conclusions of Propositions \ref{prop:barrierSTATEMENT}, \ref{prop:holder_STATEMENT}, \ref{prop:propagation} and of \Cref{cor:MDR} hold true in $Q_{r_0}$.

\begin{lemma}[Compactness and convergence to hyperplane]\label{lemma:comp_hyperplane}
	There exists a subsequence (not relabeled) such that, for almost every $t\in(-r_0^2,0]$,
	\begin{equation}\notag
	M^j_t\weakly\haus^{m}\rest S^+
	\end{equation}
	as Radon measures in $B_{r_0}$.
\end{lemma}
\begin{proof}
	By the Arzelà-Ascoli theorem,  $\gamma^j\to0$ in $C^1$ up to subsequences. By \Cref{cor:MDR}, we may apply the compactness theorems proven in \cite[Theorems 10.1 and 10.2]{whitebdry} and find a further subsequence (not relabeled) and $\Mb^\infty\in\BF(Q_{r_0},\Gamma_0)$ such that, for every $t\in(-r_0^2,0]$,
	\begin{equation}\notag
	M^j_t\weakly M^\infty_t.
	\end{equation}
	
	In particular, the weak convergence stated above and \eqref{eq:flatness_ripetuta} yield
	\begin{equation}\notag
	M^\infty_t((S^+)^\compl)=0
	\end{equation}
	for every $t\in(-r_0^2,0]$. 
	Therefore, by \Cref{def:bf}, for almost every $t$, there is an integer-valued function $\theta_t\in L^1_{loc}(S^+)$ so that
	\begin{equation}\notag
		M^\infty_t = \theta_t(\cdot)\haus^m\rest (S^+\cap B_{r_0}).
	\end{equation}
	By testing \eqref{eq:first_var} with vector fields $X\in C^1_c(B_{r_0}\setminus\Gamma_0;\RR^d)$ such that $S^\perp X=0$ everywhere, one deduces that for almost every $t$, $\theta_t(\cdot)$ is an integer-valued $W^{1,1}_{loc}$ function on $S^+$. Since $S^+\cap B_{r_0}$ is connected, $\theta_t(\cdot)$ must be constant for almost every $t$.
	Moreover, by \eqref{eq:cout_equation}, $(0,0)\in\Sigma_{\Mb^\infty}$, thus by \Cref{cout_lemma} it must be $\theta_t>0$ for every $t<0$. We conclude by remarking that, with the above remarks, for almost every $t$, $\beta_{M^\infty_t}=\theta_t\Hc^{m-1}\rest\Gamma_0$; then the assumption $M^\infty_t\in\Vc (B_{r_0},\Gamma_0)$ yields $\theta_t=1$.
\end{proof}

\newcommand{\tSigma}{\tilde\Sigma}
Before stating the next result, we define some objects that we will use in the rest of the subsection. First of all, let $F_\eps:\RR^d\to\RR^d$ be the map
\begin{equation}\notag
F_\eps(x)=\bigg(Sx,\frac{1}{\eps}S^\perp x\bigg);
\end{equation}
with a small abuse of notation, we use the same notation for the map $F_\eps\colon \RR^{d,1}\to\RR^{d,1}$ such that $F_\eps(x,t)=(F_\eps(x),t)$.
We now define
\begin{equation}\notag
\tSigma^j = F_{\eps_j}(\Sigma^j).
\end{equation}
Notice that, by \eqref{eq:flatness_ripetuta}, $\tSigma^j\subset \{(x,t)\colon|S^\perp x|\le1\}$ for every $j$.
For $j\in\NN$ and $(x',t)\in Q_{r_0}$, we define
\begin{equation}\notag
u^j(x',t) = \Big\{z\in \overline{B_1^{d-m}}\colon ((x',z),t)\in\tSigma^j\Big\};
\end{equation}
notice that such a set may well be empty or have more than one element. 
We also define $\tilde\gamma^j = F_{\eps_j}\circ\gamma^j$;
it is clear that
\begin{equation}\notag
	\tilde\gamma^j\cdot\e_{m}\to0\qquad\mbox{in }C^{1,\alpha}.
\end{equation}
Furthermore, since $||\tilde\gamma^j||_{C^{1,\alpha}(B_{r_0})}\le1$, by the Arzelà-Ascoli theorem and up to passing to a subsequence (which we do not relabel) we may find $g\colon B_{r_0}^{m-1}\to\overline{B_1^{d-m}}$ such that, for every $0<\varsigma<\alpha$, 
\begin{equation}\notag
S^\perp\tilde\gamma^j\to g\qquad\mbox{in }C^{1,\varsigma}
\end{equation}
and $\norm{g}_{C^{1,\varsigma}}\le 1$. 


In order to keep the notation light, in the following we denote by $E = \overline{B^m_{r_0}}\times\overline{B_1^{d-m}}\times[-r_0^2,0]\subset\RR^{d,1}$ and $E' = S(E) = \overline{Q^m_{r_0}}\subset\RR^{m,1}$.
We also let $E'_+ = E'\cap\{x_m\ge0\}$.
\begin{lemma}[Uniform convergence]\label{lemma:unif_conv}
	There exist a subsequence (not relabeled) and
	$u\colon E'_+\to\overline{B_1^{d-m}}$
	with the following properties:
	\begin{enumerate}[(i)]
		\item \label{concl:hausdorff} it holds
		\begin{equation}\label{eq:haus_conv1}
		d_H(\tSigma^j\cap E;\graph u)\to0
		\end{equation}
		as $j\to\infty$.
		\item \label{concl:bordo} For every $(x'',t)\in \overline{Q_{r_0}^{m-1}}$ it holds $u((x'',0),t)=g(x'')$. 
		\item \label{concl:holder} For every $X',Y'\in E'_+$,
		\begin{equation*}\label{eq:holder_cty_final}
		|u(X')-u(Y')|\le 2C_2\rho(X',Y')^\varsigma,
		\end{equation*}
		where $C_2$ and $\varsigma$ are as in \Cref{prop:holder_STATEMENT}.
	\end{enumerate}
\end{lemma}
In \eqref{eq:haus_conv1}, by $\graph u$ we mean the set $\{(x',u(x',t),t)\colon(x',t)\in E'_+\}\subset E$.
\begin{proof}
	\textbf{Step 1: Hausdorff convergence. }By \Cref{lemma:comp_hyperplane}, $\tSigma^j\cap E\neq\emptyset$ eventually. Thus one may extract a subsequence (not relabeled) so that $\tSigma^j\cap E$ converges in the Hausdorff distance to some closed set $\tSigma\subset E$. Since, by assumption, $\tSigma^j\subset\{x_{m}\ge-\eps_j\}$, it must also be $\tSigma\subset\{x_{m}\ge0\}$. We define the set-valued function
	\begin{align}
	u(x',t) =  \big\{y\in\overline{B_1^{d-m}}\colon ((x',y),t)\in\tSigma\big\}\label{eq:u}
	\end{align}
	for $(x',t)\in E'_+$. 
	
	\textbf{Step 2: $u(x',t)\neq\emptyset$ for every $(x',t)\in E'_+$. }Assume, by contradiction, that there exists $(x',t)\in E'_+\setminus S(\tSigma)$ (recall the notation $S(x,t)=(Sx,t)=(x',t)$). Then, since $S(\tSigma)$ is closed, there exists an open neighborhood $U'$ of $(x',t)$ such that $U'\subset(S(\tSigma))^\compl$. If we let $U = S^{-1}(U')\subset\RR^{d,1}$, then by \Cref{lemma:comp_hyperplane} and Fatou's lemma
	\begin{align*}
	0<\haus^{m,1}(U\cap (S^+\times\RR))\le\liminf_jM^j(U),
	\end{align*}
	thus $M^j(U)>0$ eventually. In particular, by taking smaller and smaller neighborhoods, one can pick a subsequence $j_\ell\to\infty$ and a sequence $X_\ell\in\Sigma^{j_\ell}$ so that $S(X_\ell)\to(x',t)$.	
	By using the maps $F_{\eps_j}$ defined above, we rescale in the directions of $S^\perp$ and find that, up to subsequences, there exists $z\in \overline{B_1^{d-m}}$ such that
	\begin{equation}\notag
	\tSigma^{j_{\ell}}\ni F_{\eps_{j_\ell}}(X_\ell)\to((x',z),t).
	\end{equation}
	By Step 1, $((x',z),t)\in\tSigma$, which contradicts the fact that $u(x',t)=\emptyset$.
	
	\textbf{Step 3: $u((x'',0),t)=\{g(x'')\}$. }
	Let $(x'',t)\in \overline{Q_{r_0}^{m-1}}$.
	If $y\in u((x'',0),t)$, then by Step 1 there exists a sequence $(x_j,t_j)\in\tSigma^j$ such that $x_j\to((x'',0),y)$ and $t_j\to t$.
	In particular, by \Cref{prop:barrierSTATEMENT}, it holds
	\begin{align*}\notag
	|S^\perp(x_j-\tilde\gamma^j(x_j''))|
	&= \frac{1}{\eps_j}\big|S^\perp\big(F_{\eps_j}^{-1}(x_j)-\gamma^j(x_j'')\big)\big|\\
	&\le C|x_j\cdot\e_{m}+\eps_j+\eps_j^2|^{1/2} \\
	&\longrightarrow 0
	\end{align*}
	as $j\to\infty$.
	Since $S^\perp\tilde\gamma^j$ converges uniformly to $g$ and $S^\perp x_j\to y$, it must be
	\begin{equation}\notag
	u((x'',0),t)=\{g(x'')\}.
	\end{equation}
	
	\textbf{Step 4: $u(x',t)$ is a singleton and \Cref{concl:holder} holds true. } For $i=1,2$, let $X_i = (x_i,t_i)\in \tSigma$. Let also $\rho:=\rho(S(X_1),S(X_2))$ and, without loss of generality, assume $(x_2)_{m}\ge (x_1)_{m}$. 
	
	\textit{Case 1: $(x_1)_m=0$. }By Step 1 and \Cref{prop:barrierSTATEMENT}, we have
	\begin{equation}\notag
	|S^\perp x_2- g(x_2'')|\le C (x_2)_m^{1/2}\le C\rho^{1/2}.
	\end{equation}
	Moreover, $|S^\perp x_1-g(x_2'')|=|g(x_1'')-g(x_2'')|\le C\rho$.
	Thus
	\begin{equation}\notag
	|S^\perp x_2-S^\perp x_1|\le C\rho^{1/2}+C\rho\le C\rho^\varsigma.
	\end{equation}
	
	\textit{Case 2: $(x_1)_m>0$ and $\rho=0$. }In this case, we prove that $S^\perp(x_1)=S^\perp(x_2)$. Fix $\omega$ much smaller than $(x_1)_m$. By Steps 1 and 2, we may pick $j$ large enough and three points $Y_1, Y_2, W=(w,\tau)\in\tSigma^j$ such that $\rho(X_i,Y_i)\le\omega$ and $2\omega\le\rho(S(W),S(X_i))\le 4\omega$. Up to choosing $j$ larger, we may assume that $\omega\ge C\eps_j^\varsigma$ and $(y_i)_m\ge(x_i)_m-\omega\ge2\eps_j$. Therefore, by \Cref{prop:holder_STATEMENT}, since $\rho(S(W),S(Y_i))\ge\omega$, we estimate
	\begin{align*}
	&|S^\perp(x_1-x_2)|\\
	&\qquad\le |S^\perp(x_1-y_1)|+|S^\perp(x_2-y_2)|+|S^\perp(y_1-w)|+|S^\perp(y_2-w)|\\
	&\qquad\le2\omega+C\omega^\varsigma.
	\end{align*}
	Since $\omega>0$ is arbitrary, it holds $S^\perp(x_1)=S^\perp(x_2)$. In particular, $u(x',t)$ is a singleton for every $(x',t)\in E'_+$. With a small abuse of notation, from here onwards, we will denote by $u(x',t)\in \RR^{d-m}$ the only element of the set defined in \eqref{eq:u}.
	
	\textit{Case 3: $(x_1)_m>0$ and $\rho>0$. }By Steps 1 and 2, we may choose $j$ large enough and two points $Y_1,Y_2$ such that the following hold true:
	\begin{enumerate}
		\item $Y_1,Y_2\in\tSigma^j$;
		\item for $i=1,2$, $\rho(X_i,Y_i)<\rho/8$;
		\item $C\eps_j^\varsigma \le\rho/2$;
		\item for $i=1,2$, $(y_i)_m\ge2\eps_j$.
	\end{enumerate}
	Then, by \Cref{prop:holder_STATEMENT}, it holds
	\begin{align*}
	|u(SX_1)-u(SX_2)|&\le |u(SX_1)-S^\perp(y_1)|+|u(SX_2)-S^\perp(y_2)|+|S^\perp(y_1-y_2)|\\
	&\le 2\frac{\rho}{8} + C\rho^\varsigma\le 2C\rho^\varsigma,
	\end{align*}
	as desired.
	
\end{proof}

The rest of the proof consists in proving that $u$ defined in \Cref{lemma:unif_conv} solves the heat equation in the interior of $E'_+$. To this end, we recall some facts about the heat equation. 
First, recall that $E'_+ = \overline{Q_{r_0}^m}\cap \{x_m\ge0\}$ and let us introduce the sets
\begin{gather*}
	\Int_p E'_+ = E'_+\setminus \de_pE'_+,\\
	(E'_+)_r = \{x'\in\RR^m\colon |x'|\le r_0-r\mbox{ and }x'_m\ge r\}\times[-r_0^2+r^2,0].
\end{gather*}
Notice that $\Int_p E'_+ = \bigcup_{r>0}(E'_+)_r$.

\begin{lemma}[Interior regularity for the heat equation]\label{lemma:heat_prop1}
	Let $g\in C(\de_pE'_+)$. Then there exists $h\in C^\infty(\Int_p E'_+)\cap C(E'_+)$ such that
	\begin{equation}\notag
	\begin{cases}
	\de_th-\Delta h = 0\qquad&\mbox{in }\Int_p E'_+\\
	h = g&\mbox{on }\de_pE'_+.
	\end{cases}
	\end{equation}
	Moreover, for every $r>0$ there exists $C>0$ such that, for every $(x',t)\in (E'_+)_r$, it holds
	\begin{equation}\notag
	\max\{|h(x',t)|,|\nabla h(x',t)|,|D^2h(x',t)|,|\de_th(x',t)|\}\le C\norm{g}_{L^\infty(\de_pE'_+)}.
	\end{equation}
\end{lemma}

We now proceed with the proof of \Cref{thm:IOF}.

\begin{lemma}\label{lemma:visc_sol'}
	Let $u$ be as in \Cref{lemma:unif_conv}. Then $u\in C^\infty(\Int_p E'_+;\RR^{d-m})\cap C(E'_+;\RR^{d-m})$ and
	\begin{equation}\notag
	\de_tu-\Delta u = 0
	\end{equation}
	in $\Int_pE'_+$. 
\end{lemma}
\begin{proof}	
	We take as a model the proof of \cite[Lemma 2.4]{savin2017viscosity}. We show that $u$ is equal to the solution  $h:E'_+\to \RR^{d-m}$ to the boundary value problem
	\begin{equation}\notag
	\begin{cases}
	\de_th-\Delta h = 0\qquad&\mbox{in }\Int_p E'_+\\
	h = u&\mbox{on }\de_pE'_+.
	\end{cases}
	\end{equation}
	whose existence is guaranteed by Lemma \ref{lemma:heat_prop1}. If not, there exist $r,\omega$ small and positive so that the function
	\begin{equation}\notag
	E'_+\ni (x',t)\mapsto |u(x',t)-h(x',t)|^2+\omega|x'|^2
	\end{equation}
	achieves its maximum at $(x_0',t_0)\in (E'_+)_{2r}$. Since $\tSigma^j$ converges in the Hausdorff distance to $\graph u$, for some large $j$ we may find $X_1=(x_1,t_1)\in\Sigma^j$ such that $(x'_1,t_1)\in(E'_+)_r$ and the restriction to $\Sigma^j$ of
	\begin{equation}\notag
	H(x,t) := \bigg|\frac{S^\perp x}{\eps_j}-h(Sx,t)\bigg|^2+\omega|Sx|^2
	\end{equation}
	achieves its maximum at $X_1$.
	
	We claim that, if $\eps_j$ is small enough, depending on $r$ and $\omega$, then for every $m$-dimensional subspace $T$, it holds $T:D^2H(X_1)>\de_t H(X_1)$. This would contradict \Cref{Maxprinciple}, thus concluding the proof.
	To prove the claim, we define $f(x,t) = \frac{1}{\eps_j}S^\perp x-h(Sx,t)$ and, with some straightforward computations, we write
	\begin{equation}\notag
	H(x,t) = G_1(x,t) + G_2(x,t),
	\end{equation}
	where 
	\begin{gather*}
	G_1(x,t) = |f(X_1)|^2 + 2f(X_1)\cdot(f(x,t)-f(X_1)) + \omega|Sx|^2,\\
	G_2(x,t) = |f(x,t)-f(X_1)|^2.
	\end{gather*}
	Notice that, by \Cref{lemma:heat_prop1}, there exists $C$ depending on $r$ such that
	\begin{equation}\notag
	|D^2G_1(X_1)|\le C(\omega + |f(X_1)||D^2h(SX_1)|)\le C.
	\end{equation}
	Then, just as in \cite{savin2017viscosity}, it is easy to show that, if $|T-S|\le c\omega$, then
	\begin{equation}\notag
	T:D^2G_1(X_1)>\de_tH(X_1)
	\end{equation}
	and $D^2G_2(X_1)\ge0$, thus in this case $T:D^2H(X_1)>\de_tH(X_1)$.
	On the other hand, if $|T-S|\ge c\omega$, then there exists a unit-vector $\nu\in T$ such that $S^\perp \nu\ge c\omega$. In particular, since $D^2G_2(X_1) = 2\nabla f(X_1)\nabla f(X_1)^\trasp$, it holds
	\begin{equation}\notag
	T:D^2G_2(X_1) \ge - |S\nu|^2|\nabla h(SX_1)|^2 + \frac{1}{\eps_j^2}|S^\perp\nu|^2\ge\frac{c\omega^2}{\eps_j^2}.
	\end{equation}
	We now conclude by remarking that $\de_tH(X_1) = 2f(X_1)\cdot\de_th(SX_1)\le C$ and $T:D^2G_1(X_1)\ge-|D^2G_1(X_1)|\ge-C$, thus
	\begin{equation}\notag
	T:D^2H(X_1)\ge-C+\frac{c\omega^2}{\eps_j^2}> \de_tH(X_1),
	\end{equation}
	provided $\eps_j$ is chosen small enough depending on $\omega$ and $C$ (which, in turns, is a large constant depending on $r$).
\end{proof}

Once proven that $u$ is a solution to the heat equation, it is sufficient to apply the following classical estimate:
\begin{lemma}[Boundary regularity for the heat equation]\label{lem:reg_heat_eq}
	For every $\alpha\in(0,1)$, there exist positive constants $C$ and $\beta$ with the following property. Let $u\in C^2(\Int_p E'_+)\cap C(E'_+)$ be such that
	\begin{equation}\notag
	\de_tu-\Delta u=0\qquad\mbox{in }\Int_p E'_+.
	\end{equation}
	Assume, moreover, that for all $t$, $u(\cdot,t)|_{\{x_m=0\}}=g\in C^{1,\alpha}(B_{r_0}\cap\{x_{m}=0\})$, that $|g(0)|=|Dg(0)|=0$ and that $|u|\le 1$ everywhere. Then there exists a linear operator $L:\RR^m\to \RR^{d-m}$ with $|L|\le C$ such that, for every $\eta\in(0,1/4)$, 
	\begin{equation*}\label{e:reg_heat_eq}
	|u(x',t)-L(x')|\le C\eta^{1+\beta}
	\end{equation*}
	in $(B_\eta^m\cap\{x_{m}\ge0\})\times\Tint{-\eta^2,0}$.
\end{lemma}
\begin{proof}
	See \cite[Theorem 2.1]{wangII}.
\end{proof}
\begin{remark}
	From the fact that $g\in C^{1,\alpha}$ and that $Dg(0)=0$, it follows that $L(x')=0$ if $x_{m}=0$.
\end{remark}

\begin{proof}[Conclusion of the proof of \Cref{thm:IOF}]	
	By \Cref{lemma:visc_sol'} and \Cref{lem:reg_heat_eq}, there exists $L:\RR^m\to \RR^{d-m}$ linear such that $L(x')=0$ if $x'_m=0$, $|L|\le C$ and, for every $\eta$ small, it holds
	\begin{equation}\notag
	\tSigma\cap (B^m_{2\eta}\times B_1^{d-m}\times(-4\eta^2,0])\subset\{(x,t):|S^\perp x-L(Sx)|\le C\eta^{1+\beta}\}.
	\end{equation}
	We fix $\eta$ small, to be specified later and we choose $j$ sufficiently large so that the Hausdorff distance between $\tSigma$ and $\tSigma^j$ is smaller than $\eta^{1+\beta}$. We now let $T=\{x\in\RR^d\colon S^\perp x = \eps_j L(Sx)\}$. Then it holds
	\begin{equation}\notag
	\Sigma^j\cap Q_\eta\subset \{|T^\perp x|\le C'\eps_j\eta^{1+\beta}\}.
	\end{equation}
	
	Moreover, by \Cref{prop:barrierSTATEMENT} and the fact that $|\gamma^j_m(x'')|\le \eps_j |x''|^{1+\alpha}$, it holds
	\begin{equation}\notag
	\Sigma^j\cap Q_\eta \subset\{x_m\ge-\eps_j\eta^{1+\alpha}-\eps_j^2\},
	\end{equation}
	provided $\eta\le r_2$. We choose $j$ large enough so that $\eps_j^2\le \eta^{1+\beta}$. Since $\beta$ can be chosen smaller than $\alpha$, we have
	\begin{equation}\notag
	\Sigma^j\cap Q_\eta\subset\{x_m\ge-2\eps_j\eta^{1+\beta}\}\cap\{|T^\perp x|\le C\eps_j\eta^{1+\beta}\}.
	\end{equation}
	Up to choosing $j$ larger, the above inclusion yields
	\begin{equation}\notag
	\Sigma^j\cap Q_\eta\subset\{\dist(\cdot,T^+)\le 2C\eps_j\eta^{1+\beta}\}.
	\end{equation}
	We conclude the proof by choosing $\beta'>\beta$ and $\eta$ so small that $2C\eta^{1+\beta}\le \eta^{1+\beta'}$ and we recover \eqref{eq:IOF_conclusion} (with $\beta'$ instead of $\beta$). This contradicts the assumption made at the beginning of the present subsection, thus concluding the proof.
\end{proof}

\section{Boundary behavior}\label{subsec:bdry} 
We now prove \Cref{prop:barrierSTATEMENT}. The setting is the following. Let $E_0$ and $\alpha$ be given and let $r_3$ be the constant given in \Cref{prop:propagation}. Assume $\Mb$ and $\Gamma$ satisfy the assumptions of \Cref{thm:IOF}.
Then $\mdr(\Mb,Q_{r_3})<\infty$, therefore \Cref{prop:barrierSTATEMENT} follows from the following, more general, statement:
\begin{proposition}[Boundary behavior at scale $R$]\label{prop:barrier}
	There exist $c$ and $\eps_1$ depending only on $\alpha$ with the following property. Let $0<\delta<\eps\le\eps_1$, $\Gamma\in\Flat_\alpha(\delta,B_R)$ and $\Mb\in\BF(Q_R,\Gamma)$ be such that
	\begin{equation}\notag
	\Sigma\cap Q_R\subset\{(x,t)\colon \dist(x,S^+)\le\eps R\}
	\end{equation}
	and
	\begin{equation}\label{eq:MDR_barrier}
	\mdr(\Mb, Q_R)<\infty.
	\end{equation}
	Then
	\begin{equation}\notag
	\Sigma\cap Q_{R/2}\subset\Bigg\{(x,t)\colon x_m\ge \gamma_m(x'')-R\delta^2+c R\frac{|S^\perp(x-\gamma(x''))|^2}{2(\eps R)^2}   \Bigg\}.
	\end{equation}
\end{proposition}
\begin{remark}
	The role of \eqref{eq:MDR_barrier} is to guarantee that the maximum principle (\Cref{Maxprinciple}) holds true.
\end{remark}
\begin{proof}
	By a simple rescaling argument, it is sufficient to prove the result in the case $R=1$.
	We fix $c$ small and $\eps_1\le c$, to be specified later. By contradiction, assume there exist $0<\delta\le\eps\le\eps_1$, $\Gamma$ and $\Mb$ as above, and a point $(\bar x, \bar t)\in\Sigma\cap Q_{1/2}$ such that
	\begin{equation}\notag
	0<\omega:=\frac{c}{2\eps^2}|S^\perp(\bar x-\gamma(\bar x''))|^2-\delta^2+\gamma_m(\bar x'')-\bar x_m.
	\end{equation}
	We show that, if this is the case, then we may build a family of surfaces sliding in the direction of $\e_{m}$ that touch $\Sigma$ at some point where the conclusion of \Cref{Maxprinciple} fails.
	
	In order to do so, we first define the functions $g:\RR^{d-m}\to\RR$ and $h:\RR^m\to\RR$ as
	\begin{gather*}
	g(z) = c\frac{|z-S^\perp\gamma(\bar x'')|^2}{2\eps^2}\\
	h(y)=P(y'')-|y''-\bar x''|^2-y_m,
	\end{gather*}
	where
	\begin{equation*}
		P(y'') = \gamma_m(\bar x'')+\nabla\gamma_m(\bar x'')\cdot(y''-\bar x'')-\delta^2-C|y''-\bar x''|^2,
	\end{equation*}
	and $C$ is a constant depending only on $\alpha$ chosen so that
	\begin{equation}\label{eq:sviluppo}
	P(x'')\le\gamma_m(x'')
	\end{equation}
	(to show that such $C$ depending only on $\alpha$ exists, use the fact that $\gamma\in C^{1,\alpha}(B_R)$ and Young's inequality).
	Then, choose a smooth function $f\colon\RR\to\RR$ such that $f(-1)=-4c$, $f|_{t\ge-1/4}\ge-\frac{\omega}{2}$, $f<0$ everywhere and $f'(t)\le 8c$ everywhere.
	We now set
	\begin{equation}\notag
	H(x,t)=g(S^\perp x)+h(Sx)+f(t).
	\end{equation}
	This way, the zero-level set of $H$ is a surface sliding in the $\e_m$-direction.
	Notice that
	\begin{equation}\label{eq:sitocca!}
	H(\bar x,\bar t)
	=\omega+f(\bar t)>0.
	\end{equation}
	We now show that, if $(x,t)\in\Sigma\cap((\Gamma\times\RR)\cup\de_p Q_{1})$, then $H(x,t)\le0$.	
	\begin{enumerate}
		\item \label{case:-1} If $x\in\Sigma_{-1}$, then
		\begin{itemize}
			\item $g(S^\perp x)\le c(\eps+\delta)^2/(2\eps^2)=2c$, since $|S^\perp x|\le\eps$ and $|\gamma(x'')|\le\delta\le\eps$;
			\item by \eqref{eq:sviluppo}, $h(Sx)\le\gamma_m(x'')-x_m\le 2\eps$.
		\end{itemize}
		The two above facts, along with the assumption $f(-1) = -4c$, yield
		\begin{equation}\notag
		H(x,-1)\le 2c + 2\eps -4c\le0
		\end{equation}
		provided $\eps\le c$.
		
		\item If $x\in\de {B_{1}}\cap \Sigma_t$, then
		$|S^\perp x|\le\eps$ and $x_m\ge-\eps$, thus
		\begin{equation}\notag
		|x''|\ge\sqrt{1-\eps^2-x_m^2}\ge\frac{3}{4}-x_m,
		\end{equation}
		provided $\eps$ is small enough. In particular, $|x''-\bar x''|\ge\frac{1}{4}-x_m$. Hence:
		\begin{itemize}
			\item since $\norm{\gamma}_{C^{1,\alpha}(B_1)}\le\delta$, we have
			\begin{equation*}
			h(Sx)\le 2\delta -(C+1)|x''-\bar x''|^2-x_m\le 2\delta - (C+1)\bigg(\frac{1}{4}-x_m\bigg)^2-x_m;
			\end{equation*}
			\item as in \Cref{case:-1}, $g(S^\perp x)\le 2c$;
			\item $f(t)\le0$;
		\end{itemize}
		Therefore
		\begin{align*}
		H(x,t)&\le 2c+2\delta - (C+1)\bigg(\frac{1}{4}-x_m\bigg)^2-x_m\le2c+2\delta-\frac{C}{4(1+C)}\le0
		\end{align*}
		provided $C\ge1$ and $c,\delta$ are small enough.
		
		\item Lastly, for every $x\in\Gamma$ and $t\in(-1,0)$, under the assumptions $\delta\le\eps$ and $c\le1$, it holds 
		\begin{align*}\notag
		g(S^\perp x)&=\frac{c}{2\eps^2}|S^\perp(\gamma(x'')-\gamma(\bar x''))|^2\\
		&\le \frac{c}{2\eps^2}\norm{\nabla \gamma}_\infty^2|x''-\bar x''|^2\\
		&\le |x''-\bar x''|^2.
		\end{align*}
		Since $f\le0$ and $h(Sx)\le\gamma_m(x'')-|x''-\bar x''|^2-x_m$, we have
		\begin{equation}\notag
		H(x,t)\le \gamma_m(x'')-x_m=0.
		\end{equation}
	\end{enumerate}
	
	Points 1-3 above and \eqref{eq:sitocca!} show that there must exist $Y = (y,s)\in Q_1\cap\Sigma$ with $y\notin\Gamma$ such that $H|_{\{t\le s\}}$ has a local maximum at $(y,s)$.
	
	We now show that one can choose $c$ even smaller, if needed, so that the existence of such a point would contradict the maximum principle. Indeed, since $|S^\perp y|\le\eps$, if $c$ is small enough then $|\nabla h(S y)|^2/|\nabla g(S^\perp y)|^2\ge\eps$, thus
	\begin{equation}\notag
	|S^\perp\nabla H(Y)| = |\nabla g(S^\perp y)| \le(1-\eps)|\nabla H(Y)|.
	\end{equation}
	Therefore, if $T$ is a $m$-dimensional subspace of $\RR^d$ such that $T\perp \nabla H(Y)$, then
	\begin{equation}\notag
	T\sprod S^\perp \ge \eps
	\end{equation}
	and
	\begin{equation}\notag
	T\sprod D^2 H(Y)
	=T\sprod \begin{pmatrix}
	D^2h&0\\
	0&D^2g
	\end{pmatrix}
	\ge -|D^2h(Sy)| + \eps|D^2g(S^\perp y)|.
	\end{equation}
	Now, simple computations show that, up to multiplications by constants depending only on $m$ and $\alpha$, $|D^2h(Sy)|\le 1$ and $|D^2g(S^\perp y)|\ge\frac{c}{\eps^2}$.
	Therefore, if $\eps$ is much smaller than $c$, then $T\sprod D^2H(y)\ge\frac{c}{2\eps}$. However, by \Cref{Maxprinciple}, it holds
	\begin{equation}\notag
	\inf_{T\perp\nabla H(Y)}T\sprod D^2H(Y)\le\de_tH(Y)=f'(s)\le 8c,
	\end{equation}
	which is a contradiction.
\end{proof}

\section{Decay of oscillations: proof of \texorpdfstring{\Cref{prop:holder_STATEMENT}}{the decay of oscillations}}\label{subsec:osc_decay}
\newcommand{\backpball}{\mathcal{B}}
\newcommand{\backparab}{\mathcal{P}}
In the present section, we prove \Cref{prop:holder_STATEMENT}. 

We begin by giving the following definition:
\begin{definition}
	Let $u:\RR^{m,1}\to[-\infty,1]$ be an upper-semicontinuous function. Assume that, whenever a smooth function $\phi:\RR^{m,1}\to\RR$ touches $u$ from above at some $(x_0',t_0)\in U\times I$ (according to the terminology set in \Cref{subsec:MAX}) and $|\nabla\phi(x_0',t_0)|,|D^2\phi(x_0',t_0)|$ are smaller than some fixed universal constant $\delta_0$, then
	\begin{equation}\label{eq:model}
		\de_t\phi-\Mc^+(D^2\phi)\le0
	\end{equation}
	at $(x'_0,t_0)$ (see \Cref{subsec:MAX} for the definition of $\Mc^+$). Then $u$ is said to be a {viscosity subsolution to \eqref{eq:model}} in $U\times I$.
\end{definition}
The reader should notice that the classical definition of viscosity solution is slightly different than ours, in that the test function $\phi$ usually has no restrictions on the magnitude of $|\nabla\phi|$ and $|D^2\phi|$ at the touching point.

The proof of \Cref{prop:holder_STATEMENT} is achieved in three steps:
\begin{enumerate}
	\item First of all, one sees that the support of a $\Mb$ behaves, in some sense, like the graph of a viscosity subsolution to \eqref{eq:model}, as in the definition above; this was proved in \Cref{cor:max_princ2}.
	\item By exploiting the results in \cite{wangsmall}, one shows that, if a $\Sigma$ has a point far enough from $S$, then the mass of $\Mb$ near that point cannot be too small.
	\item If $\Sigma$ does not have the decay of oscillations stated in \Cref{prop:holder_STATEMENT}, then by the previous step the mass of $\Mb$ in some parabolic cylinder must be large; this contradicts the small density assumption \eqref{eq:double_dens_0}.
\end{enumerate}



Before proceeding, we introduce some notations that we are going to use in the present subsection.
Given $\theta\in(0,1)$, we define the set
\begin{equation}\notag
\backparab_1^{\theta}=\bigg\{(x',t)\in \RR^{m,1}\colon|x'|^2<\frac{-t}{\theta^2}<1\bigg\}.
\end{equation}
One should compare these sets with those which, in \cite{wangsmall}, are called \enquote{parabolic balls}. Our definition slightly differs from theirs; notice that with our choice $\backparab_1^\theta\subset B^m_1\times(-\theta^2,0)$.

\begin{lemma}[Measure estimate, \cite{wangsmall}]\label{lemma:measure}
	For every $\theta>0$ and $\mu\in(0,1)$, there exist small constants $\eta', r$ with the following property. Let $u\colon \RR^{m,1}\to[-\infty,1]$ be a
	viscosity subsolution to \eqref{eq:model}
	in $B_1^m\times(-\theta^2,0)$ and assume that
	\begin{equation}\notag
	u(Y_0)\ge1-\eta'
	\end{equation}
	for some $Y_0\in B^m_{r}\times(-\theta^2r^2,0)$. Then
	\begin{equation}\label{eq:thisiswang}
	\Lc^{m,1}(\{u\ge1-\mu\}\cap\backparab_1^{\theta})\ge(1-\mu)\Lc^{m,1}(\backparab_1^{\theta}).
	\end{equation}
\end{lemma}
\begin{proof}
	This result corresponds, essentially, to \cite[Lemma 4.3]{wangsmall}.
	Apart from some trivial adjustment of constants, there are two caveats:
	\begin{itemize}
		\item The results in \cite{wangsmall} are stated with the classical definition of viscosity solutions, where no bound on the test function at the touching point is required. However, it is easy to see that the results are valid for our definition of viscosity solution, as well.
		
		\item In our setting, we allow $u$ to be merely upper-semicontinuous and, possibly, take infinite values, while in \cite{wangsmall} $u$ is required to be continuous. This minor point can be easily overcome by looking at the sup-convolution of $u$:
		\begin{equation}\notag
			u_\delta(x,t) = \sup\bigg\{u(y,s)-\frac{1}{\delta}(|x-y|^2+(t-s)^2)\bigg\},	
		\end{equation}
		which conserves the property of being a viscosity subsolution to \eqref{eq:model} and for which \eqref{eq:thisiswang} holds true, by \cite[Lemma 4.3]{wangsmall}. Letting $\delta\searrow0$ gives the desired conclusion.
	\end{itemize}
\end{proof}


Before stating the next result, we fix some further notations. For any closed set $\Sigma\subset\RR^{d,1}$ and any $\Omega\subset\RR^{d,1}$, we let
\begin{equation}\notag
\osc (\Sigma, \Omega) = \inf\bigg\{h>0\colon \mbox{there is $y\in\RR^d$ such that }\Sigma\cap\Omega\subset\{x\colon |S^\perp(x-y)|\le h\} \bigg\}.
\end{equation}
We also let
\begin{equation}\notag
C_r = \big\{x\in\RR^d\colon |Sx|<r\big\}.
\end{equation}

\begin{lemma}[Harnack inequality]\label{lemma:harnack_0}
	For every $\delta \in(0,1)$, there exist small constants $\eps_2,\theta,r,\eta$ with the following property.
	Let $\eps\le\eps_2$ and $\Mb\in\BF(C_1\times\Tint{-\theta^2,0})$ be such that:
	\begin{gather}
	\Sigma\subset\{|S^\perp x|\le\eps\},\label{eq:piattezza000}\\
	\int_{C_1}\hk(\cdot,t)\,dM_t\le 2-\delta\qquad\mbox{for all }t\in(-\theta^2,0),\label{e:densityass1}
	\end{gather}
	and
	\begin{equation}\label{eq:hp_unif_bound_harnack1}
	\mdr(\Mb, C_1\times\Tint{-\theta^2,0})<\infty.
	\end{equation}
	Then
	\begin{equation}\label{eq:harn_0_concl}
	\osc(\Sigma,C_r\times\Tint{-\theta^2r^2,0})\le(1-\eta)\eps.
	\end{equation}
\end{lemma}

The proof of the above result involves some technical estimates.
It is therefore convenient to give an overview of the strategy. If \eqref{eq:harn_0_concl} does not hold, then one finds two points $Y_1$ and $Y_2$ in $\Sigma$ that are far enough in $S^\perp$. By applying \Cref{lemma:measure} twice, we find that in $C_1\times(-\theta^2,0)$ the mass of $\Mb$ must be almost that of two $m$-dimensional disks. This contradicts \eqref{e:densityass1}, which encodes the fact that the mass of $\Mb$ must not exceed by too much that of a single disk.

\begin{proof}[Proof of \Cref{lemma:harnack_0}]	
	Let $\delta\in(0,1)$ be given. Fix $\theta$ and $\mu$, which we will specify later, and let $r$ and $\eta'$ be chosen accordingly as in \Cref{lemma:measure}. Moreover, fix $\eps$ much smaller than $\mu$ and $\eta\le\eta'$, to be specified later. Assume, by contradiction, that there exist $\Mb\in\BF(C_1\times\Tint{-\theta^2,0})$ that satisfies the assumptions of the present result with the choices made above, and two points $Y_1=(y_1,s_1),Y_2=(y_2,s_2)\in \Sigma\cap(C_r\times\Tint{-\theta^2r^2,0})$ with $|S^\perp y_1-S^\perp y_2|\ge 2(1-\eta)\eps$.
	For every $(x',t)\in B_1^m\times\Tint{-\theta^2,0}$ and for $i=1,2$, let
	\begin{equation}\notag
	u_i(x',t) = \frac{1}{2\eps}\sup\Big\{|z-S^\perp y_i|\colon z\in S^\perp \mbox{ and }(x',z)\in\Sigma_t\Big\}.
	\end{equation}
	Notice that $u_1$ and $u_2$ are upper-semicontinuous and, for every $(x',t)$, either $u_1(x',t),u_2(x',t)\in[0,1]$ or $u_1=u_2=-\infty$.
	By \Cref{cor:max_princ2} and \eqref{eq:hp_unif_bound_harnack1}, both $u_1$ and $u_2$ are viscosity subsolutions to \eqref{eq:model}.
	Moreover,
	\begin{equation}\notag
	u_1(Sy_2,s_2)\ge\frac{1}{2\eps}	\big|S^\perp y_2-S^\perp y_1\big|\ge 1-\eta\ge1-\eta',
	\end{equation}
	hence, by \Cref{lemma:measure},
	\begin{equation*}\label{eq:stima_misura1}
	\Lc^{m,1}(\{u_1\ge 1-\mu\}\cap\backparab_1^\theta)\ge(1-\mu)\Lc^{m,1}(\backparab_1^\theta).
	\end{equation*}
	With the same argument, one also obtains
	\begin{equation}\label{eq:stima_misura2}
		\Lc^{m,1}(\{u_2\ge 1-\mu\}\cap\backparab_1^\theta)\ge(1-\mu)\Lc^{m,1}(\backparab_1^\theta).
	\end{equation}
	
	We now want to estimate
	\begin{equation*}
		\int_{C_1\times(-\theta^2,0)}\hk\,dM.
	\end{equation*}
	We first define, for $i=1,2$, the sets
	\begin{equation}\notag
	A_i = \{(x,t)\in\RR^{d,1}\colon (Sx,t)\in\backparab_1^{\theta}, |S^\perp (x-y_i)|\le\eps/2 \mbox{ and }t\le-2\eps^2/\delta\}.
	\end{equation}
	Notice that $A_1\cap A_2=\emptyset$ and, by \eqref{eq:piattezza000}, for $M$-a.e. $(x,t)\in A_i$, it holds
	\begin{align*}
	\hk(x,t)&=\exp\bigg(\frac{|S^\perp x|^2}{4t}\bigg)\hk'(Sx,t)\\
	&\ge\exp\bigg(-\frac{\eps^2}{8\eps^2/\delta}\bigg)\hk'(Sx,t) \\
	&\ge\bigg(1-\frac{\delta}{8}\bigg)\hk'(Sx,t),
	\end{align*}
	where $\hk'(x',t)\colon=\hk((x',0),t)$.
	
	Therefore we have
	\begin{align}
		&\int_{C_1\times(-\theta^2,0)}\hk\,dM\notag\\
		&\qquad\ge\int_{A_1}\hk\,dM + \int_{A_2}\hk\,dM\notag\\
		&\qquad\ge\bigg(1-\frac{\delta}{8}\bigg)
		\Bigg(\int_{A_1}\hk'(Sx,t)\,dM(x,t)+
		\int_{A_2}\hk'(Sx,t)\,dM(x,t)\Bigg).\label{eq:inizio_stima}
	\end{align}
	
	Moreover, by \Cref{lemma:federer} and by the coarea formula,
	\begin{align}
	\int_{A_i}\hk'(Sx,t)\,dM(x,t)
	\ge\int_{S(A_i\cap\Sigma)}\hk'(x',t)\,d\Lc^{m,1}(x',t).\label{eq:last_side}
	\end{align}
	
	We may assume that $\mu$ and $\eta$ are smaller that some universal constant so that,
	if $z\in\RR^{d-m}$ with $|z|\le\eps$ is such that $\frac{1}{2\eps}|z-S^\perp y_2|\ge1-\mu$, then
	\begin{equation*}
		|z-S^\perp y_1|\le \frac{\eps}{2}.
	\end{equation*}
	In particular, we have
	\begin{gather*}
	S(A_1\cap\Sigma)\supset\{u_2\ge1-\mu\}\cap\backparab_1^\theta\cap\{t\le-2\eps^2/\delta\}
	\end{gather*}
	which, together with \eqref{eq:stima_misura2}, yields that $S(A_1\cap\Sigma)$ covers a large portion of $\backparab_1^\theta$: namely
	\begin{align*}\notag
	\Lc^{m,1}\bigg(S(A_1\cap\Sigma)\bigg)
	&\ge \Lc^{m,1}\bigg(\backparab_1^\theta\cap\{u_2\ge1-\mu\}\cap\{t\le-2\eps^2/\delta\}\bigg)\\
	&\ge \Lc^{m,1}\bigg(\backparab_1^\theta\cap\{u_2\ge1-\mu\}\bigg) - \frac{2\eps^2}{\delta}\\
	&\ge (1-2\mu)\Lc^{m,1}\big(\backparab_1^\theta\big),
	\end{align*}
	provided $\eps^2\le c\delta\mu$ for some $c$ small universal.
	
	We are now ready to choose $\mu$, depending on $\delta$, so that the above inequality and the fact that $\hk\in L^1(\Lc^{m,1}\rest\backparab_1^\theta)$ yield
	\begin{align}
	&\int_{S(A_1\cap\Sigma)}\hk'\,d\Lc^{m,1}\notag\\
	&\qquad\ge\int_{\backparab_1^\theta}\hk'\,d\Lc^{m,1}-\frac{\delta\theta^2}{8}\notag\\
	&\qquad = \theta^2\int_{B_1^m}\hk'(\cdot,-\theta^2)\,d\Lc^m-\frac{\delta\theta^2}{8}.\label{eq:eps,mu_small}
	\end{align}
	Finally, we also choose $\theta$ small such that
	\begin{equation}\label{eq:theta_small}
	\int_{B_1^m}\hk'(\cdot,-\theta^2)\,d\Lc^m\ge\int_{\RR^m}\hk'(\cdot,-\theta^2)\,d\Lc^m-\frac{\delta}{8} = 1-\frac{\delta}{8}.
	\end{equation}
	By \eqref{eq:eps,mu_small} and \eqref{eq:theta_small}, it holds
	\begin{equation}\label{eq:almost_there}
		\int_{S(A_1\cap\Sigma)}\hk'\,d\Lc^{m,1}\ge\theta^2\bigg(1-\frac{\delta}{4}\bigg).
	\end{equation}
	The same argument can be repeated for $A_2$, thus giving
	\begin{equation}\label{eq:almost_there2}
		\int_{S(A_2\cap\Sigma)}\hk'\,d\Lc^{m,1}\ge\theta^2\bigg(1-\frac{\delta}{4}\bigg).
	\end{equation}

	We conclude the proof by combining \eqref{eq:inizio_stima}, \eqref{eq:last_side}, and \eqref{eq:almost_there}, \eqref{eq:almost_there2}, obtaining
	\begin{align*}
	&\int_{C_1\times(-\theta^2,0)}\hk\,dM\\
	&\qquad\ge 2\theta^2\bigg(1-\frac{\delta}{8}\bigg)\bigg(1-\frac{\delta}{4}\bigg)\\
	&\qquad\ge \theta^2\bigg(2-\frac{3}{4}\delta\bigg)
	\end{align*}
	which contradicts \eqref{e:densityass1}. This concludes the proof.
\end{proof}

%
%

A simple rescaling argument allows one to iterate Lemma \ref{lemma:harnack_0} and obtain the following

\begin{proposition}\label{prop:holder_better}
	For every $\delta\in(0,1)$ there exist $C$ (large) and $\varsigma$ (small) with the following property. Let $\Mb\in\BF(C_R\times(-R^2,R^2))$ be such that
	\begin{equation}\label{eq:holder_better_1}
	\mdr(\Mb, C_R\times(-R^2,R^2))<\infty
	\end{equation}
	and assume that, for every $(x,t)\in C_{R/2}\times(-R^2/4,R^2/4))$ and every $s\in(t-R^2/4,t)$, it holds
	\begin{equation}\label{eq:holder_better_2}
	\int_{C_{R/2}(x)}\hk(\cdot-x,s-t)\,dM_s\le 2-\delta.
	\end{equation}
	If $\eps=\osc(\Sigma,C_R\times(-R^2,R^2))$, then for any couple $(x,t),(y,s)\in C_{R/2}\times(-R^2/4,R^2/4)\cap\Sigma$ such that $\rho=\rho(X',Y')\ge CR^{1-\varsigma}\eps^\varsigma$, it holds
	\begin{equation*}\label{eq:osc_da ottenere}
	|S^\perp(x-y)|\le C \eps\bigg(\frac{\rho}{R}\bigg)^\varsigma.
	\end{equation*}
\end{proposition}
\begin{proof}
	We prove the result for $R=1$, as the general case follows by replacing $\Mb$ with $\Dc_{R}\Mb$.
	
	Let $\eps_2, \theta, r, \eta$ be the constants given in \Cref{lemma:harnack_0} in correspondence to $\delta$.
	Without loss of generality, we may assume that $\eps\le\eps_2$, otherwise the result follows by choosing $C$ large enough.
	Consider the rescaled flows $\Mb^k = \Dc_{r^k}(\Mb-X)$. By induction, the assumptions of \Cref{lemma:harnack_0} are in place for every integer $k$ such that
	\begin{equation}\label{eq:interi}
	\bigg(\frac{1-\eta}{r}\bigg)^k\eps\le\eps_2.
	\end{equation}
	Therefore, scaling back to the original flow, we see that for those $k$:
	\begin{equation}\notag
	\osc\big(\Sigma_{\Mb},C_{r^k}(x)\times\Tint{t-\theta^2r^{2k},t}\big)\le (1-\eta)^k\eps.
	\end{equation}
	
	Let now $X=(x,t)$ and $Y=(y,s)$ be two points in $C_{1/2}\times(-1/4,1/4)\cap\Sigma$ and let $\rho= \rho((x',t),(y',s))$.
	Without loss of generality, we may assume that $t\ge s$. 
	Furthermore, by taking $C\ge2/\theta$, we may clearly reduce ourselves to the case $\rho\le\theta/2$.
	By choosing $\varsigma$ small enough and $C$ larger than the choice made above, if necessary, we infer from $\rho\ge C\eps^\varsigma$ that there exists $k\in\NN$ satisfying \eqref{eq:interi} such that $r^{k+1}\le2\rho/\theta\le r^k$. Thus
	\begin{equation}\notag
	Y\in C_{2\rho}(x)\times\Tint{t-4\rho^2,t}\subset C_{r^k}(x)\times\Tint{t-\theta^2r^{2k},t}
	\end{equation}
	hence it must be $|S^\perp(x-y)|\le 2(1-\eta)^k\eps$. We conclude the proof by taking $C$ larger and $\varsigma$ smaller, if needed, so that $2(1-\eta)^k\le C\rho^\varsigma$.
\end{proof}

We finally prove \Cref{prop:holder_STATEMENT}.
\begin{proof}[Proof of \Cref{prop:holder_STATEMENT}]
	Let $r_2 = \frac{1}{2}\min\{r_1,r_3\}$, where $r_1$ and $r_3$ are given in Propositions \ref{prop:barrierSTATEMENT} and \ref{prop:propagation}, respectively. Let also $X=(x,t), Y=(y,s)$ be two points in $\Sigma\cap Q_{r_2}$. Without loss of generality, we assume that $R:=x_{m}\ge y_{m}\ge 2\eps$. Let $\rho=\rho((x',t),(y',s))$; finally, let $\varsigma$ be the constant determined in \Cref{prop:holder_better} corresponding to $\delta=\frac{1}{2}$. We shall distinguish two cases.
	
	If $\rho\le\frac{R}{8}$, then we may find $t'\in\Tint{-r_2^2,0}$ such that $X,Y\in C_{R/4}(x)\times(t'-R^2/16,t'+R^2/16)$ and $C_{R/4}(x)\subset\Gamma^\compl$. Since $R\le1$, the assumption $\rho\ge C\eps^\varsigma$ yields $\rho\ge CR^{1-\varsigma}\eps^\varsigma$. By \Cref{prop:propagation} and \Cref{cor:MDR}, \eqref{eq:holder_better_1} and \eqref{eq:holder_better_2} hold true, thus \Cref{prop:holder_better} applies and we obtain
	\begin{equation}\notag
	|S^\perp(x-y)|\le C\bigg(\frac{\rho}{R}\bigg)^\varsigma\osc(\Sigma,\mathcal{U}(X)),
	\end{equation}
	where $\mathcal{U}(X) \colon= C_{R/4}(x)\times(t'-R^2/16,t'+R^2/16)$. By \Cref{prop:barrierSTATEMENT}, we may estimate
	\begin{align*}
	\osc(\Sigma,\mathcal{U}(X))\le 2 C\eps\sqrt{2R+\eps+\eps^2}+CR\norm{\nabla\gamma}_{\infty}\le C\eps R^{1/2}, 
	\end{align*}
	since $\eps\le R/2$ and $\norm{\nabla\gamma}_\infty\le C\eps$. Thus
	\begin{equation}\notag
	|S^\perp(x-y)|\le C\eps R^{1/2-\varsigma}\rho^\varsigma\le C\eps\rho^\varsigma,
	\end{equation}
	since $\varsigma$ can be chosen smaller than $1/2$.
	
	On the other hand, if $\rho\ge\frac{R}{8}$, then it is sufficient to use \Cref{prop:barrierSTATEMENT} twice and the fact that $\norm{\nabla\gamma}_\infty\le C\eps$ to estimate:
	\begin{gather*}
	|S^\perp(x-\gamma(x''))|\le C\eps (R+\eps+\eps^2)^{1/2}\le C\eps\rho^{1/2},\\
	|S^\perp(y-\gamma(y''))|\le C\eps (2R+\eps+\eps^2)^{1/2}\le C\eps\rho^{1/2},\\
	|S^\perp(\gamma(y'')-\gamma(x''))|\le C\eps \rho.
	\end{gather*}
	We therefore conclude
	\begin{equation}\notag
	|S^\perp(x-y)|\le C\eps\rho^{1/2}\le C\eps\rho^\varsigma,
	\end{equation}
	which is the desired result. 
\end{proof}

\section{\texorpdfstring{$C^{1,\beta}$}{H\"{o}lder} regularity}\label{sec:final}

In the present section, we prove the following $\eps$-regularity theorem:

\begin{theorem}[$C^{1,\beta}$ regularity]\label{thm:final}
	For every $E_0$ and $\alpha$, there are small constants $\eps_3,\Lambda,\eta$ and $\beta$ with the following property. Let $\eps\le\eps_3$, $\Gamma\in\Flat_\alpha(\eps,B_1)$, $\Mb\in\BF(B_1\times[-\Lambda,0],\Gamma)$ be such that $(0,0)\in\Sigma_{\Mb}$,
	\begin{gather*}
	\Sigma_\Mb\subset\{(z,\tau)\colon \dist(z,S^+)\le\eps\},\\
	\sup_{t\in[-\Lambda,0]}M_t(B_1)\le E_0
	\end{gather*}
	and
	\begin{equation*}
	\int_{B_1}\hk(\cdot,-\Lambda)\,dM_{-\Lambda}\le\frac{3}{4}.
	\end{equation*}	
	Then there is $u\in C^{1,\beta}(Q^m_\eta,\RR^{d-m})$ with $||u||_{C^{1,\beta}}\le C\eps$ such that $\Sigma_\Mb\cap Q^d_\eta=\graph u$ and, for all $t\in(-\eta^2/4,0]$ it holds $\de\Sigma_t\cap B_\eta \subset \Gamma$.
\end{theorem}

Before proving the above result, we record the following consequence of \Cref{thm:IOF}
\begin{proposition}[Iteration of the improvement of flatness]\label{prop:iof_iter}
	Under the assumptions of \Cref{thm:IOF}, for every $X=(x,t)\in\Sigma\cap Q_\eta$:
	\begin{itemize}
		\item If $x\in\Gamma$, then there exists a $m$-dimensional half plane $T^+_X$ such that $\de T^+_X=T_x\Gamma$ and 
		\begin{equation}\notag
		\Sigma\cap Q_{\eta^k}(X)\subset\{\dist(\cdot,x+T^+_{X})\le 2\eps\eta^{k(1+\beta)}\}
		\end{equation}
		for every $k\in\Nb$;
		\item If $x\notin\Gamma$, then there exists a $m$-dimensional plane $T_X$ such that 
		\begin{equation}\notag
		\Sigma\cap Q_{\eta^k}(X)\subset\{\dist(\cdot,x+T_{X})\le 2\eps\eta^{k(1+\beta)}\}
		\end{equation}
		for every $k\in\Nb$.
	\end{itemize}
\end{proposition}
\begin{proof}
	This result is a straightforward consequence of an iteration of \Cref{thm:IOF}. Namely, given $X\in\Sigma\cap Q_\eta$ with $x\in\Gamma$, we may find a sequence of half-planes $T^+_k$ such that
	\begin{equation*}
		\Sigma\cap Q_{\eta^k}\subset\{\dist(\cdot, T_k^+)\le \eta^{k(1+\beta)}\eps\}.
	\end{equation*}
	Moreover, $|T_k^+-T^+_{k-1}|\le C\eps\frac{\eta^{k(1+\beta)}}{\eta^k}$ for some $C$ depending only on $E_0$ and $\alpha$.
	Therefore, $\{T_k^+\}$ converges to some half plane $T_X^+$ for which the conclusion of the proposition holds true. 
	
	For the case $x\notin\Gamma$, one may see \cite{Tonegawabook} or replicate the techniques of the previous sections.
 \end{proof}

\begin{remark}\label{rk:decay_0}
	Given $x\in\Gamma$ and $T_{(x,t)}^+$ as in \Cref{prop:iof_iter}, throughout the rest of the present section, we let $T_{(x,t)}$ be the $m$-dimensional plane obtained by reflecting $T_{(x,t)}^+$ across $T_x\Gamma$.	
	We remark the following conclusion of \Cref{thm:IOF}: there is $C$ depending only on $E_0$ and $\alpha$ such that, for every $X\in\Sigma\cap Q_\eta$, it holds
	\begin{equation*}
		|T_{X}-S|\le C \eps.
	\end{equation*}
\end{remark}

We are now ready to prove \Cref{thm:final}.
\begin{proof}[Proof of \Cref{thm:final}]
	Up to a rotation, we may assume, without loss of generality, that $T_{(0,0)}$ defined in \Cref{prop:iof_iter} coincides with the plane $S$ that satisfies the assumptions of the present result.
		
%

	\textbf{Step 1: $\Sigma$ is the graph of a $C^{1,\beta}$ function over $S(\Sigma)$. }Let $X\in\Sigma\cap Q_\eta^d$ and, for simplicity of notation, let $T = T_{X}$ as defined in \Cref{rk:decay_0}; recall that $|T-S|\le C\eps$.
	For any other point $Y\in\Sigma\cap Q_\eta^d$, we may write
	\begin{align*}
	|S^\perp(x-y)|&\le |T^\perp(x-y)|+|S^\perp-T^\perp||x-y|\\
	&\le C\eps\rho(X,Y)^{1+\beta} + C\eps\rho(X,Y)\\
	&\le 2C\eps\rho(X,Y).
	\end{align*}
	If $\eps_3$ is smaller than some universal constant, we conclude
	\begin{equation*}
	|S^\perp(x-y)|\le 3 C\eps\rho(SX,SY).
	\end{equation*}
	The above inequality, together with \Cref{prop:iof_iter}, yields
	\begin{equation}\label{eq:almeno_lip}
	|T^\perp(x-y)|\le C\eps\rho(SX,SY)^{1+\beta}.
	\end{equation}
	
	Now, by using the identities $S+S^\perp= T+T^\perp=\Id$, it may be checked by direct computations that
	\begin{equation}\label{eq:direct_computations}
	(\Id-S^\perp T)S^\perp(x-y)-S^\perp T S(x-y)= S^\perp T^\perp (x-y).
	\end{equation}
	Since $|S-T|\le C\eps$, $(\Id-S^\perp T)$ is invertible and $|(\Id-S^\perp T)^{-1}|\le2$ provided $\eps$ is small enough. In particular, by letting $L=(\Id-S^\perp T)^{-1}S^\perp$, we have $|L|\le 2$. Then \eqref{eq:almeno_lip} and \eqref{eq:direct_computations} above give
	\begin{equation}\label{eq:infatti_holder}
	\big|S^\perp x-S^\perp y - L T(Sx-Sy)    \big|= |L T^\perp(x-y)|\le 2|T^\perp(x-y)|\le C\eps\rho(SX,SY)^{1+\beta}.
	\end{equation}
	
	This proves that, indeed, there is $u\in C^{1,\beta}(S(\Sigma\cap Q_\eta);\RR^{d-m})$ with $||u||_{C^{1,\beta}}\le C\eps$ such that $\Sigma\cap Q_\eta = \graph u$.
	
	\textbf{Step 2: Absence of holes. }
	We now \enquote{split} the parabolic cylinder $Q_{\eta/2}$ into two components, on two opposite sides of $\Gamma$. To this end, we define the following sets:
	\begin{equation*}
		E'_+ := \{x'\in B_{\eta/2}^m\colon x'_m>\gamma_m(x'') \},\qquad E'_-:= \{x'\in B_{\eta/2}^m\colon x'_m<\gamma_m(x'') \},
	\end{equation*}
	their parabolic counterparts
	\newcommand{\Ec}{\mathcal{E}}
	\begin{equation*}
		\Ec'_\pm := E'_\pm\times(-\eta^2/4,0]\subset\RR^{m,1}
	\end{equation*}
	and, lastly, 
	\begin{equation*}
		\Ec_{\pm} = \{(x,t)\in\RR^{d,1}\colon (x',t)\in E_{\pm}\}.
	\end{equation*}
	Arguing for the positive side (as the argument applies for the other case) we claim that, if $X_1=(x_1,t_1)\in\Sigma\cap\Ec_+$, then
	\begin{equation}\label{eq:no_holes}
		\Ec'_+\cap\{t\le t_1\}\subset S(\Sigma\cap Q_\eta).
	\end{equation} 
	To prove this, assume by contradiction that there is $(x_0',t_0)\in \Ec'_+\setminus S(\Sigma)$ with $t_0<t_1$. Since $\gamma\in C^{1,\alpha}$, it is easy to see that there exist a smooth curve $p:[t_0,t_1]\to E'_+$ and $\rho>0$ with the following properties:
	\begin{equation*}
	\begin{cases}
		p(t_0)=x'_0\mbox{ and } p(t_1)=x'_1;\\
		Q^m_\rho(p(t),t)\subset \Ec'_+;\\
		Q^m_\rho(p(t_0),t_0)\subset S(\Sigma)^\compl.
	\end{cases}
	\end{equation*}
	The fact that $\Sigma$ is closed and \Cref{cout_lemma} yield the existence of a time $\bar t\in(t_0,t_1)$ such that $Q^m_\rho(p(t),t)\subset S(\Sigma)^\compl$ for every $t< \bar t$ and a point $(y_0,s_0)\in\Sigma$ such that $(y_0',s_0)\in\de B_\rho^m(p(\bar t))\times[\bar t-\rho^2,\bar t]$.
	
	Let us now consider a sequence $r_j\searrow0$ and define the dilations
	\begin{equation}\notag
		\Mb^j = \Dc_{r_j}(\Mb-(y_0,s_0)).
	\end{equation}
	Since $\Mb$ has bounded maximal density ratio, the compactness theorems in \cite[Section 10]{whitebdry} yield that, up to passing to a subsequence, $\Mb^j$ converges to a limit Brakke flow $\Mb^\infty$. 
	
	Then \eqref{eq:infatti_holder}
	implies that there exists a $m$-dimensional half plane $T^+$ such that
	\begin{equation*}
		\Sigma_{\Mb^\infty}\subset T^+\times(-\infty,0].
	\end{equation*}
	Moreover, since $(y_0,s_0)\in\Sigma_\Mb$, we have $(0,0)\in\Sigma_{\Mb^\infty}$.
	
	We finally show that this violates the maximum principle. Up to a change of coordinates, say $T^+=\{x_{m+1}=\dots=x_d=0\mbox{ and }x_m>0\}$ and let
	\begin{equation}\notag
		f(x,t)= \frac{|T^\perp x|^2}{2}-\frac{|x''|^2}{2m}+\frac{|x_m|^2}{2}-x_m+\frac{1}{2m}t,
	\end{equation}
	where $x'' = (x_1,\dots,x_{m-1})$.
	Then $f|_{\Sigma_{\Mb^\infty}\cap \{t\le0\}}$ has a local maximum at $(0,0)$. However, it holds $\de_tf(0,0) = \frac{1}{2m}$ and
	\begin{equation*}
		\inf_{\substack{T\in\Gr(m,d)\\T\perp\nabla f(0,0)}} T\sprod D^2f(0,0) = 1-(m-1)\frac{1}{m} = \frac{1}{m},
	\end{equation*}
	which contradicts \Cref{Maxprinciple}, thus proving \eqref{eq:no_holes}.

	\textbf{Step 3: Conclusion. }Since, by assumption, $(0,0)\in\Sigma_\Mb$, by \Cref{cout_lemma} it must be $M_t(B_r)>0$ for every $t<0$ and $r>0$. However, if there were $t_1<0$ and $r>0$ such that
	\begin{equation*}
		M_{t_1}(B_r\cap\{x_m<\gamma_m(x'')\})>0,
	\end{equation*}
	by the previous step it should be $\Ec'_-\cap\{t\le t_1\}\subset S(\Sigma\cap Q_\eta)$. However, this would contradict \Cref{prop:barrier} with $R=\eta$, provided $\eps_3$ is chosen small enough. Therefore we have that, for every $t<0$ sufficiently close to $0$, it must be
	\begin{equation*}
		M_t(B_r\cap\{x_m>\gamma_m(x'')\})>0
	\end{equation*}
	and, by the previous step,
	$\Ec'_+\subset S(\Sigma\cap Q_\eta)$ which, together with \eqref{eq:infatti_holder} amounts to saying that there exists $u:\Ec'_+\to\RR^{d-m}$ such that
	\begin{equation*}
		\Sigma\cap (\Ec_+\cup\Ec_-) = \graph u = \{(x,t)\colon Sx\in\Ec'_+ \mbox{ and }S^\perp x = u(Sx,t)\}
	\end{equation*}
	and $||u||_{C^{1,\beta}}\le C\eps$. 
	
	We only have to prove that $\de\Sigma_t \cap B_{\eta/2} \subset \Gamma$ for every $t\in (-\eta^2/4,0]$. If there were $x\in \de\Sigma_t \cap B_{\eta/2}\setminus \Gamma$, then by the fact that $u\in C^{1,\beta}$ there would be some blow-up of $\Sigma$ around $(x,t)$ that is contained in a $m$-dimensional half-plane for all times. Arguing as in the previous step, one finds a contradiction to \Cref{Maxprinciple}.
\end{proof}

\appendix
\section{Proof of \texorpdfstring{\Cref{lemma:brutto}}{Lemma \ref{lemma:brutto}}}\label{sec:proofofLemmabrutto}
Up to rescaling and translating, it is sufficient to prove that there exist $A$ and $\Lambda$ small and positive, depending only on $\delta$ and $\alpha$, such that, if $\Gamma$ is a $(m-1)$-dimensional properly embedded submanifold of $B_2$ with $[\Gamma]_{C^{1,\alpha}(B_2)}\le A$, then
\begin{equation}\label{stimaquesto}
\int_{-\Lambda}^0\int|T_y\Gamma^\perp\nabla\hk_{1}(y,t)|\,d\Gamma(y)\,dt\le\frac{1}{2}\chi_{\Gamma^\compl}(0)+\delta.
\end{equation}
For brevity, we denote by $\Gamma_y$ the space $T_y\Gamma$.
Throughout the proof, $C$ will denote constants (possibly changing from one expression to another) depending only on $m,d,\alpha$.

\textbf{Case 1: $0\in\Gamma$. }
We start by remarking that
\begin{align*}
&\int_{-\Lambda}^0\int|\Gamma_y^\perp\nabla\hk_{1}(\cdot,t)|\,d\Gamma\,dt\\
&\qquad\le\int_{-\Lambda}^0\int_{B_2}|\Gamma_y^\perp\nabla\hk(\cdot,t)|\,d\Gamma(y)\,dt+C\Gamma(B_2)\int_{-\Lambda}^0\frac{e^{1/(4t)}}{(-t)^{m/2}}\,dt.
\end{align*}
If $A$ is smaller than some universal constant, then $\Gamma(B_2)\le C$, thus we may take $\Lambda$ small depending on $\delta$ so that the last term in the above inequality is smaller than $\delta/2$.
Therefore we reduce ourselves to proving that, if $A$ is small, then
\begin{align*}
I_1:=\int_{-\Lambda}^0\int_{B_2}|\Gamma_y^\perp\nabla\hk(\cdot,t)|\,d\Gamma\,dt
\le\frac{\delta}{2}.
\end{align*}
Since $[\Gamma]_{C^{1,\alpha}(B_2)}\le A$ is small, for every $(y,t)\in\Gamma\times(-\infty,0)$:
\begin{equation}\notag
|\Gamma_y^\perp\nabla\hk(y,t)|\le C\frac{e^{|y|^2/t}}{(-t)^{1+m/2}}|\Gamma_y^\perp y|\le CA\frac{e^{|y|^2/t}}{(-t)^{1+m/2}}|y|^{1+\alpha}.
\end{equation}
We then use the fact that, if $A$ is smaller than some universal constant, then $\Gamma\cap B_2$ is the graph over $\Gamma_0$ of some function $\gamma:\RR^{m-1}\to\RR^{d-m+1}$ such that $\norm{\gamma}_{C^{1,\alpha}}\le CA$. In particular, by using the area formula and the fact that the $||\nabla\gamma||_{L^\infty(B_2)}\le 1$ for $A$ small enough, we obtain
\begin{align*}
\int_{B_2}|\Gamma_y^\perp\nabla\hk(y,t)|\,d\Gamma(y)\le CA\frac{1}{(-t)^{1+m/2}}\int_{\RR^{m-1}}|y|^{1+\alpha}e^{|y|^2/t}\,d\Lc^{m-1}(y) = CA (-t)^{\frac{\alpha}{2}-1}
\end{align*}
for some $C$ depending only on $m$ and $\alpha$.
Therefore, assuming $\Lambda\le1$,
\begin{equation}\notag
I_1\le CA\int_{-1}^0(-t)^{\frac{\alpha}{2}-1}\,dt\le CA.
\end{equation}
We conclude the proof in the case $0\in\Gamma$ by choosing $A\le\frac{\delta}{2C}$.

\textbf{Case 2: $0\notin\Gamma$. }Let $E_\Gamma$ be the $m$-dimensional Hausdorff measure restricted to the exterior cone
\begin{equation}\notag
C_\Gamma:= \{\lambda y\colon\lambda\ge1\mbox{ and }y\in\Gamma\}
\end{equation}
with multiplicity, as defined in \cite[Section 7]{whitebdry}. With similar computations to those in the proof of \cite[Theorem 7.1]{whitebdry}, we may show that
\begin{align}
&\int_{-\Lambda}^{0}\int|\Gamma_y^\perp\nabla\hk_{1}(y,t)|\,d\Gamma(y)\,dt\notag\\
&\qquad\le -\lim_{\tau\nearrow0}\int\hk_1(\cdot,\tau)\,dE_\Gamma + \int\hk_1(\cdot,-\Lambda)\,dE_\Gamma + C\Lambda E_\Gamma(B_2)\notag\\
&\qquad = \int\hk_1(\cdot,-\Lambda)\,dE_\Gamma + C\Lambda E_\Gamma(B_2),\label{eq:caso_fuori}
\end{align}
Where the last equality comes from the fact that $0\notin\Gamma$.

In order to prove \eqref{stimaquesto}, we argue by contradiction: assume there is a sequence $\{\Gamma^j\}$ with $0\notin\Gamma^j$ such that $\norm{\Gamma^j}_{C^{1,\alpha}(B_2)}\le \frac{1}{j}$ for which the left-hand side of \eqref{eq:caso_fuori} is greater than $\frac{1}{2}+\delta$. One may show that, up to extracting a subsequence, $E_{\Gamma^j}$ converges weakly to $\haus^{m}\rest S^+$, where $S^+$ is some $m$-dimensional half plane such that $0\notin \Int(S^+)$. Therefore
\begin{equation}\notag
\limsup_{j\to\infty}\Bigg\{\int\hk_1(\cdot,-\Lambda)E_{\Gamma^j}+C\Lambda E_{\Gamma^j}(B_2)\Bigg\}
\le \int_{S^+}\hk_1(\cdot,-\Lambda)\,d\haus^m+C\Lambda\haus^{m}(\overline{B_2}\cap S^+).
\end{equation}
Since $0\notin\Int(S^+)$, the integral in the right-hand side of the above inequality is smaller than $\frac{1}{2}$ for every choice of $\Lambda$. On the other hand, $\Lambda$ may be chosen so small that
\begin{equation*}
	C\Lambda\haus^{m}(\overline{B_2}\cap S^+)\le\frac{\delta}{2},
\end{equation*}
which contradicts the assumption made above, thus concluding the proof.

\section{Proof of \texorpdfstring{\Cref{lemma:federer}}{Lemma \ref{lemma:federer}}}\label{app:proof_of_federer}
We refer the reader to \cite[Lemma 9.4]{kasai_tonegawa} for a detailed proof of \Cref{lemma:federer}. Since some minor modifications are needed, in this section we sketch the outline of the proof.

Let $U\subset\RR^d$ be open, $I\subset\RR$ be a non-empty interval, $\Gamma$ be a $(m-1)$-dimensional $C^{1,\alpha}$ submanifold of $U$ and let $\Mb\in\BF(U\times I, \Gamma)$. We assume that $\Mb$ satisfies a bound of the form
\begin{equation}\label{eq:bound_ultimo}
\mdr(\Mb,U\times I)\le E_1<\infty
\end{equation}
and we let $\Sigma=\Sigma_\Mb$ be its space-time support.

Before proceeding, by virtue of \Cref{cout_lemma}, we fix small constants $c_1$, $c_2$ and $R_0$, depending on $E_1$ and $\Gamma$, such that for every $(x,t)\in \Sigma$ and every $R\le R_0$ such that $B_R(x)\times(t-c_1R^2,t)\compact U\times I$, it holds
\begin{equation}\notag
	M_{t-c_1R^2}(B_{R/2}(x))\ge c_2 R^m.
\end{equation}

By \Cref{def:bf}, for almost every $t\in I$ there exist a $m$-dimensional rectifiable set $E\subset U$ and a positive, integer valued function $\theta_t:E_t\to\Nb$ such that $M_{t}=\theta_t(\cdot)\haus^m\rest E_t$
We choose a time $t$ as above, with the additional condition that $s\mapsto M_s(\phi)$ is continuous at $t$ for every $\phi\in C_c(U)$. By \cite[Proposition 3.3]{Tonegawabook}, almost every $t\in I$ satisfies the latter condition.

We claim that, for every such $t$ and for every $B_{3r}(x_0)\compact U$, it holds
\begin{equation}\label{eq:tonegawa_tesi}
\haus^m((\Sigma_t\setminus E_{t})\cap B_r(x_0))=0;
\end{equation}
this clearly implies \eqref{federer_tesi}.

In order to prove \eqref{eq:tonegawa_tesi}, we argue by contradiction. Assume that there is $(x_0,t_0)\in U\times I$ and $r>0$ such that $B_{3r}\compact U$ and $\haus^m(A\cap B_r(x_0))>0$. Without loss of generality, we may take $x_0=0$, $t=0$ and set $A:=\Sigma_0\setminus E_0$.

Let
\begin{equation}\notag
	A_k := \big\{x\in A\cap B_r\colon M_0(B_R(x))\le c_2{R^m}/{2}\mbox{ for all }R\in(0,{r}/{k})\big\}.
\end{equation}
Since, for $\haus^m$-a.e. $x\in E_0^\compl$, it holds
\begin{equation*}
	\lim_{R\searrow 0}\frac{M_0(B_R(x))}{R^m}=0,
\end{equation*}
we have
\begin{equation*}
	0<\haus^m(A\cap B_r)=\haus^m(\bigcup_{k\in\Nb}A_k).
\end{equation*}
Therefore we may find $k\in\Nb$ such that $b_0:=\haus^m(A_k)>0$.

By standard measure-theoretic arguments, it is not hard to show that there exists $c$ small universal such that, for every $R$ small enough, we may find $N\in\Nb$ and a finite collection of points $\{x_j\}_{j=1}^N\subset A_k$ such that $\{B_R(x_j)\}$ are mutually disjoint and
\begin{equation}\label{eq:stima_strana}
NR^m\ge cb_0.
\end{equation}

By definition of $A$, since $x_j\in A_k$, we have
\begin{equation}\label{eq:poca_massa_in0}
M_0\bigg(\bigcup_{j=1}^NB_R(x_j)\bigg)\le N c_2\frac{R^m}{2}.
\end{equation}
On the other hand, by \Cref{cout_lemma} and the fact that $x_j\in \Sigma_0$, we have
\begin{equation}\label{eq:tanta_massa_prima}
M_{-c_1R^2}\bigg(\bigcup_{j=1}^NB_{R/2}(x_j)\bigg)\ge N c_2 R^m.
\end{equation}

We now fix a cut-off function $\phi\in C^\infty_c(B_1)$ such that $\phi\in[0,1]$ everywhere, $\phi|_{B_{1/2}}\equiv1$ and $|\nabla\phi|\le 4$. Then, given $R$ small, we let $\phi_0(x)=\phi(x/(2r))$, $\phi_j(x)= \phi((x-x_j)/R)$ and
\begin{equation}\notag
\tilde\phi = \phi_0-\sum_{j=1}^N\phi_j.
\end{equation}
Then clearly $\tilde\phi\in[0,1]$ everywhere and $|\nabla\tilde\phi|\le C/R$. Notice, moreover, that
\begin{equation}\notag
\sum_{j=1}^N\chi_{B_{R/2}(x_j)}\le\sum_{j=1}^N\phi_j\le\sum_{j=1}^N\chi_{B_R(x_j)}
\end{equation}
For brevity, set $s=-c_1R^2$. By \eqref{eq:poca_massa_in0} and \eqref{eq:tanta_massa_prima}, we have
\begin{align}
M_0(\phi_0)-M_s(\phi_0)&= \big(M_0(\tilde\phi)-M_s(\tilde\phi) \big)+ \bigg(M_0\big(\sum\phi_j\big)-M_s\big(\sum\phi_j\big)\bigg)\notag\\
&\le\big(M_0(\tilde\phi)-M_s(\tilde\phi) \big) + (Nc_2R^m/2-Nc_2R^m)\notag\\
&\le \big(M_0(\tilde\phi)-M_s(\tilde\phi) \big) - c_3 b_0\label{eq:prosegui0}
\end{align}
for some $c_3$ small, where \eqref{eq:stima_strana} was used in the last inequality.

We now estimate, by using \Cref{def:bf},
\begin{align}
M_0(\tilde\phi)-M_s(\tilde\phi)&\le \int_{s}^0\int H\cdot\nabla\tilde\phi\,dM_t\,dt\notag\\
&\le \bigg(\int_{s}^0\int_{B_{2r}} |H|^2\bigg)^{1/2}\bigg(\int_{s}^0\int|\nabla\tilde\phi|^2\bigg)^{1/2}.\label{eq:prosegui}
\end{align}
By \eqref{eq:bound_ultimo} and the fact that $s=-c_1R^2$, we have, for some $C$ large,
\begin{equation}\notag
\int_{s}^0\int|\nabla\tilde\phi|^2\,dM_t\,dt
\le(-s) ||\nabla\tilde\phi||^2_\infty M_t(B_{2r})\le C E_1 r^m,
\end{equation}
therefore \eqref{eq:prosegui0} and \eqref{eq:prosegui} yield
\begin{equation}\notag
\left(\int_{-c_2R^2}^0\int_{B_{2r}} |H|^2\,dM_t\,dt\right)^\frac{1}{2}\ge\left(\frac{1}{CE_1r^m}\right)^\frac{1}{2}\bigg(M_0(\phi_0)-M_{-c_1R^2}(\phi_0) +c_3 b_0\bigg).
\end{equation}
By assumption, $t\mapsto M_t(\phi_0)$ is continuous at $0$. Thus we may choose $R$ so small that the right-hand side of the above inequality is larger than $c\frac{b_0}{(E_1r^m)^{1/2}}$ for some $c$ small enough. 

Finally, we consider the function $\hat\phi= \phi(x/(3r))$. By \Cref{def:bf} and the Cauchy-Schwarz inequality, we have
\begin{align*}
M_0(\hat\phi)-M_{-c_1R^2}(\hat\phi)
&\le \int_{-c_1R^2}^{0}\int_{B_{3r}}\big(-\hat\phi|H|^2+\nabla\hat\phi\cdot H\big)\,dM_t\,dt\\
&\le -\int_{-c_1R^2}^{0}\int_{B_{3r}}\frac{1}{2}\hat\phi|H|^2\,dM_t\,dt+\int_{-c_1R^2}^{0}\int_{B_{3r}}\frac{|\nabla\hat\phi|^2}{2\hat\phi}\,dM_t\,dt\\
&\le -\frac{1}{2}\int_{-c_1R^2}^{0}\int_{B_{2r}}|H|^2\,dM_t\,dt + CE_1r^mR^2\\
&\le -c\frac{b_0^2}{E_1r^m}
\end{align*}
provided $R$ is chosen small enough. This contradicts the continuity of $t\mapsto M_t(\hat\phi)$ at $0$, thus concluding the proof.

\printbibliography
\end{document}